\documentclass[11pt]{article}

\usepackage[T1]{fontenc}
\usepackage{lmodern}
\usepackage{amsmath,amssymb,amsthm,mathtools}
\usepackage{enumitem}
\usepackage{microtype}
\usepackage{cite}
\usepackage[hidelinks]{hyperref}
\usepackage[a4paper,margin=1in]{geometry}

\newtheorem{theorem}{Theorem}[section]
\newtheorem{proposition}[theorem]{Proposition}
\newtheorem{lemma}[theorem]{Lemma}
\newtheorem{corollary}[theorem]{Corollary}
\theoremstyle{definition}
\newtheorem{definition}[theorem]{Definition}
\newtheorem{example}[theorem]{Example}
\theoremstyle{remark}

\newcommand{\R}{\mathbb{R}}

\title{\textbf{Super-\(g\)-H\"older Regularity in Stieltjes Dynamics:\\
Atomic Structure, Exact Extremal Values, and Recovery of the Stieltjes Clock}}
\author{%
\large Serkan \.{I}lter$^{1,*}$ \qquad
\large H\"ulya Duru$^{2}$ \qquad
\large Seyit Koca$^{3,4}$\\[7pt]
\normalsize $^{1,2}$Department of Mathematics, Faculty of Science, Istanbul University, Istanbul, T\"urkiye\\
\normalsize $^{3}$Department of Management Information Systems, Istinye University, Istanbul, T\"urkiye\\
\normalsize $^{4}$Institute of Science, Istanbul University, Istanbul, T\"urkiye\\
\\[5pt]
\small \texttt{ilters@istanbul.edu.tr} \qquad
\texttt{hduru@istanbul.edu.tr}\\
\small \texttt{seyit.koca@istinye.edu.tr}\\[5pt]
\footnotesize $^{*}$Corresponding author: Serkan \.{I}lter, \texttt{ilters@istanbul.edu.tr}%
}
\date{}

\begin{document}
\maketitle

\begin{abstract}
Unlike ordinary time, a Stieltjes clock may advance continuously, remain constant over intervals, or jump. For an ordinary continuous clock, H\"older regularity with exponent \(\alpha>1\) forces constancy, whereas jumps of a Stieltjes clock may allow nonconstant behavior. We refer to H\"older regularity of exponent \(\alpha>1\) measured relative to \(g\) as super-\(g\)-H\"older regularity, and show that, in finite dimensions, such functions admit an atomic representation determined entirely by their jumps, with their \(g\)-H\"older seminorms given exactly by the corresponding normalized jump sizes.

This atomic representation yields a linear isometric description of the associated function space, together with its Banach and separability properties, an exact total-variation formula with an optimal bound, and finite-jump approximation results.

We then apply this structure to Stieltjes differential equations. Nonzero state jumps require positive clock jumps, leading to quantitative finite-jump bounds and, for affine dynamics, exact thresholds for the number of jumps and extremal terminal distances under a prescribed total Stieltjes mass. Finally, we study recovery of the Stieltjes clock from observed state jumps. For known single-valued dynamics, clock jumps can be recovered under a natural local identifiability condition, while for nonconvex differential inclusions the super-\(g\)-H\"older bound can reduce, and in some cases remove, nonuniqueness.
\end{abstract}



\section{Introduction}

A system need not experience time as a smooth and uninterrupted flow. It may evolve continuously, remain unchanged for a period, or change abruptly. Stieltjes calculus represents these possibilities through a nondecreasing clock \(g\), which may advance continuously, stay constant on intervals, or jump. This raises a basic question: how does regularity change when increments are measured by \(g\) rather than by ordinary elapsed time?

The distinction becomes decisive for H\"older exponents greater than one. In ordinary continuous time, a H\"older function with exponent \(\alpha>1\) on an interval must be constant. A Stieltjes clock changes this conclusion because a jump cannot be divided into arbitrarily small clock increments. Consequently, nonconstant \(g\)-H\"older behavior with exponent \(\alpha>1\) can occur only through the discontinuous part of the clock.

Regularity with respect to a Stieltjes derivator has been studied through \(g\)-continuity, \(g\)-absolute continuity, Stieltjes differentiation, compactness properties, chain rules, and Stieltjes integral equations \cite{FernandezTojoVillanueva2024Compactness,MarquezAlbesSlavik2024Chain}. Existence and uniqueness results, numerical approximation, systems with stationary periods and jumps, and the kernel of the Stieltjes derivative have also been investigated \cite{LopezPousoMarquez2019Systems,MarquezAlbesTojo2021,LopezPousoMarquez2019Models,FernandezTojo2020,FernandezEtAl2025Kernel}. More recently, Tojo related Stieltjes \(L_g^p\) and Sobolev spaces to classical function spaces through a generalized inverse of the derivator and used this correspondence to connect Stieltjes differential equations with ordinary differential equations \cite{AdrianTojo2025Connection}. For set-valued dynamics, related results include periodic problems, viability and Filippov-type statements, and relaxation for nonconvex differential inclusions \cite{MarraffaSatco2022Periodic,SatcoSmyrlis2023,MarraffaSatco2023,MarraffaSatco2023Infinite}. These results provide the Stieltjes setting in which we examine the exponent range \(\alpha>1\).

Let \(I=[a,b]\subset\R\), with \(a<b\), and let \(g:I\to\R\) be a left-continuous nondecreasing Stieltjes derivator. We study functions \(f:I\to\R^d\) satisfying \(g\)-H\"older regularity with exponent \(\alpha>1\), which we refer to as super-\(g\)-H\"older regularity. Our main question is: how do the continuous and atomic parts of \(g\) determine the structure of this function class, and what consequences does this structure have for Stieltjes dynamics?

The structural result goes beyond the observation that jumps can support nonconstant behavior. For \(\alpha>1\), the nonatomic part of the Stieltjes clock contributes no nonconstant variation: every super-\(g\)-H\"older function admits a pure-jump representation determined by the jumps of \(g\). Moreover, its \(g\)-H\"older seminorm is recovered exactly from the corresponding normalized jump amplitudes. Thus the regularity condition reduces the function to a normalized jump sequence indexed by the jump points of the clock.

Classical bounded-variation results already give precise relations between one-sided jumps and variation; see, in particular, Chistyakov's results for metric-space-valued maps of bounded Jordan variation \cite{Chistyakov1998BV}. The additional super-\(g\)-H\"older assumption with exponent \(\alpha>1\) turns this jump information into an exact description of the function and its \(g\)-H\"older seminorm. This provides the structural input for the dynamical, extremal, and inverse problems studied later.

The atomic representation also determines the function-space structure. After fixing the initial value, the super-\(g\)-H\"older space is linearly isometric to \(\ell^\infty(D_g;\R^d)\), with the normalized jump amplitudes serving as coordinates. Without fixing the initial value, the corresponding space is linearly isometric to \(\R^d\oplus_1\ell^\infty(D_g;\R^d)\). This immediately gives the Banach structure and shows that the space is separable exactly when \(D_g\) is finite. The same representation also yields a sharp bounded-variation estimate, finite-jump approximation, and the compactness results developed in Section~3.

The atomic representation has direct consequences for Stieltjes differential equations in finite-dimensional state spaces. In the super-\(g\)-H\"older regime, the continuous part of the clock cannot support nonzero state variation, so state changes are confined to the jump points of \(g\). If the vector field is uniformly bounded away from zero at these nonzero state jumps, each such jump requires a positive minimum clock increment. A prescribed total Stieltjes mass can therefore support only finitely many nonzero state jumps.

For affine dynamics, the jump constraints become explicit. We obtain exact thresholds for the number of nonzero state jumps compatible with a prescribed total Stieltjes mass. For a fixed number of jumps, we then determine the minimum and maximum terminal distance over the admissible jump sizes. Standard equalization and majorization arguments provide useful comparison points \cite{BoydVandenberghe2004,MarshallOlkinArnold2011}; here the lower bound for each jump size depends on the preceding jumps. The extremal analysis also shows that, under the same total Stieltjes mass, the strongest admissible atomic evolution remains strictly separated from the corresponding continuous-time response.

The atomic structure also leads naturally to an inverse problem: which parts of a Stieltjes clock can be reconstructed from the observed state trajectory? For known single-valued dynamics, every observed nonzero state jump determines the corresponding clock jump whenever the local vector field is nonzero. Stieltjes mass may remain undetected where the dynamics vanish, so complete recovery requires that the vector field not vanish along the observed trajectory. In particular, observations of individual state jumps can determine the atomic clock up to an additive normalization, whereas terminal observations alone may correspond to different Stieltjes clocks. For nonconvex differential inclusions, further nonuniqueness appears because several admissible velocities may generate the same observed jump with different jump sizes. The super-\(g\)-H\"older bound restricts these candidates and can, in some cases, restore uniqueness.

The paper is organized as follows. Section~2 introduces the Stieltjes notation and the super-\(g\)-H\"older setting. Section~3 establishes the atomic representation and the resulting function-space properties. Section~4 applies this structure to Stieltjes differential equations and bounds the number of nonzero state jumps. Sections~5 and~6 treat thresholds for the number of jumps and extremal terminal distances from the equilibrium for affine dynamics. Section~7 studies recovery of a Stieltjes clock for known single-valued dynamics, while Section~8 considers the corresponding problem for nonconvex differential inclusions. Section~9 summarizes the main conclusions.

\section{Preliminaries and the Stieltjes clock}

Let \(I=[a,b]\subset\R\), with \(a<b\). Throughout, let \(g:I\to\R\) be a left-continuous nondecreasing Stieltjes derivator. When emphasizing its role in parametrizing the evolution, we refer to \(g\) as the Stieltjes clock. The Lebesgue--Stieltjes measure generated by \(g\) is denoted by \(\mu_g\). For \(a\le s<t\le b\), we use the convention
\[
\mu_g([s,t))=g(t)-g(s).
\tag{2.1}
\]
For every \(\tau\in[a,b)\), define
\[
\Delta^+g(\tau):=g(\tau+)-g(\tau),
\tag{2.2}
\]
and
\[
D_g:=\{\tau\in[a,b):\Delta^+g(\tau)>0\}.
\tag{2.3}
\]
Whenever a function \(f:I\to\R^d\) has a right limit at such a point \(\tau\), we write
\[
\Delta^+f(\tau):=f(\tau+)-f(\tau).
\]
Then
\[
\mu_g(\{\tau\})=\Delta^+g(\tau).
\tag{2.4}
\]
For \(\tau\in D_g\), we write \(m_\tau:=\Delta^+g(\tau)=\mu_g(\{\tau\})>0\) for the mass of the atom at \(\tau\).

We decompose the Stieltjes measure as
\[
\mu_g=\mu_g^c+\mu_g^a,
\tag{2.5}
\]
where \(\mu_g^c\) is the nonatomic part of \(\mu_g\) and
\(\mu_g^a\) is its purely atomic part. Thus, for every \(t\in I\),
\[
\mu_g^c(\{t\})=0.
\]

Since the atoms of \(\mu_g\) are precisely the jump points of
\(g\), and the mass of the atom at \(\tau\in D_g\) is
\(\Delta^+g(\tau)\), the atomic part is given by
\[
\mu_g^a
=
\sum_{\tau\in D_g}
\Delta^+g(\tau)\,\delta_\tau,
\tag{2.6}
\]
where \(\delta_\tau\) denotes the Dirac measure concentrated at
\(\tau\). Consequently, for every Borel set \(E\subset I\),
\[
\mu_g^a(E)
=
\sum_{\tau\in D_g\cap E}
\Delta^+g(\tau).
\]

Because \(g\) is nondecreasing on the compact interval \(I\), the set
\(D_g\) is at most countable and
\[
\sum_{\tau\in D_g}\Delta^+g(\tau)
\le
g(b)-g(a)
<
\infty.
\]
Hence the series defining \(\mu_g^a\) is well defined.

Here and below, the term ``continuous part'' refers to the nonatomic
component \(\mu_g^c\). It does not imply absolute continuity with
respect to Lebesgue measure; in general, \(\mu_g^c\) may also contain
a singular continuous component.

A function \(f:I\to\R^d\) is said to be \(g\)-continuous at \(t\in I\) if, for every \(\varepsilon>0\), there exists \(\delta>0\) such that \(\|f(s)-f(t)\|<\varepsilon\) whenever \(s\in I\) and \(|g(s)-g(t)|<\delta\). The function \(f\) is \(g\)-continuous on \(I\) if it is \(g\)-continuous at every \(t\in I\).

A function \(f:I\to\R^d\) is \(g\)-absolutely continuous if, for every \(\varepsilon>0\), there exists \(\delta>0\) such that, for every finite family of pairwise disjoint intervals \((s_k,t_k)\subset I\), the condition \(\sum_k|g(t_k)-g(s_k)|<\delta\) implies \(\sum_k\|f(t_k)-f(s_k)\|<\varepsilon\).

For \(f:I\to\R^d\) and \(\alpha>0\), define
\[
[f]_{\alpha,g}:=
\sup_{\substack{s,t\in I\\g(s)\ne g(t)}}
\frac{\|f(t)-f(s)\|}{|g(t)-g(s)|^\alpha},
\tag{2.7}
\]
with the convention that the supremum is zero if the indexing set is empty. We say that \(f\) is \(g\)-H\"older continuous with exponent \(\alpha>0\) if \([f]_{\alpha,g}<\infty\) and \(f(t)=f(s)\) whenever \(s,t\in I\) satisfy \(g(t)=g(s)\). Equivalently, there exists \(L\ge0\) such that \(\|f(t)-f(s)\|\le L|g(t)-g(s)|^\alpha\) for every \(s,t\in I\). In particular,
\[
g(t)=g(s)\quad\Longrightarrow\quad f(t)=f(s).
\tag{2.8}
\]
When \(\alpha>1\), we call this \emph{super-\(g\)-H\"older regularity} and refer to \(f\) as a \emph{super-\(g\)-H\"older function}.

\subsection{Function spaces and auxiliary notation}

Throughout, \(\mathbb N=\{1,2,\ldots\}\) and \(\mathbb N_0=\{0,1,2,\ldots\}\). On Euclidean spaces, \(\|\cdot\|\) denotes the Euclidean norm and \(\langle\cdot,\cdot\rangle\) the Euclidean inner product. For a set \(E\), \(\operatorname{card}(E)\) denotes its cardinality. For a bounded map \(h:I\to\R^d\), we write \(\|h\|_\infty:=\sup_{t\in I}\|h(t)\|\).

For fixed \(\alpha>1\) and \(d\ge1\), let \(\mathcal H_{\alpha,g}(I;\R^d)\) denote the space of all super-\(g\)-H\"older maps \(f:I\to\R^d\) with exponent \(\alpha\), equipped with the norm
\[
\|f\|_{\alpha,g}:=\|f(a)\|+[f]_{\alpha,g},
\]
so that \(\mathcal H_{\alpha,g}(I;\R^d)\) is a normed linear space. We also use the subspace
\[
\mathcal H_{\alpha,g}^{0}(I;\R^d)
:=
\{f\in\mathcal H_{\alpha,g}(I;\R^d):f(a)=0\},
\]
on which \([\cdot]_{\alpha,g}\) is a norm.

We write
\[
\ell^\infty(D_g;\R^d)
=
\left\{
\xi:D_g\to\R^d:
\sup_{\tau\in D_g}\|\xi(\tau)\|<\infty
\right\},
\]
endowed with the norm \(\|\xi\|_{\ell^\infty}=\sup_{\tau\in D_g}\|\xi(\tau)\|\). If \(D_g=\varnothing\), this space is understood as the zero space. We also write \(c_0(D_g;\R^d)\) for the subspace of all \(\xi\in\ell^\infty(D_g;\R^d)\) such that, for every \(\varepsilon>0\), the set \(\{\tau\in D_g:\|\xi(\tau)\|\ge\varepsilon\}\) is finite.

For normed spaces \(X\) and \(Y\), the notation \(X\oplus_1Y\) denotes the product space \(X\times Y\) equipped with the norm \(\|(x,y)\|_{X\oplus_1Y}=\|x\|_X+\|y\|_Y\). Accordingly, \(\R^d\oplus_1\ell^\infty(D_g;\R^d)\) consists of all pairs \((x,\xi)\) with \(x\in\R^d\) and \(\xi\in\ell^\infty(D_g;\R^d)\), endowed with the norm \(\|(x,\xi)\|=\|x\|+\|\xi\|_{\ell^\infty}\). For \(r\in\R\), \(\lfloor r\rfloor\) denotes the greatest integer not exceeding \(r\).

For \(f:I\to\R^d\), where \(I=[a,b]\), the Jordan total variation of \(f\) on \(I\) is
\[
\operatorname{Var}_I(f)
:=
\sup\left\{
\sum_{j=1}^{n}\|f(t_j)-f(t_{j-1})\|:
 n\in\mathbb N,\;
 a=t_0<t_1<\cdots<t_n=b
\right\}.
\]
The supremum is taken over all finite partitions of \(I\). We say that \(f\) is of bounded variation on \(I\) if \(\operatorname{Var}_I(f)<\infty\).

Define also the flat-clock set
\[
C_g:=\{t\in I:\text{$g$ is constant on a neighbourhood of $t$}\}.
\tag{2.9}
\]
It is standard that
\[
\mu_g(C_g)=0.
\tag{2.10}
\]
For \(t\in I\setminus(D_g\cup C_g)\), the Stieltjes derivative is defined, whenever the corresponding limit exists, by
\[
f'_g(t)
=
\lim_{\substack{s\to t\\ g(s)\ne g(t)}}
\frac{f(s)-f(t)}
     {g(s)-g(t)}.
\tag{2.11}
\]

At a jump point \(\tau\in D_g\), equivalently at an atom of \(\mu_g\), the Stieltjes derivative is given by the corresponding jump quotient:
\[
f'_g(\tau)
=
\frac{\Delta^+f(\tau)}
     {\Delta^+g(\tau)}.
\tag{2.12}
\]
This is the standard convention used for left-continuous nondecreasing derivators; see, for example, \cite{FernandezTojo2020}. Since \(C_g\) is \(\mu_g\)-null, no value of the derivative on \(C_g\) is needed for the almost-everywhere statements or the integral representation used below.

For \(g\)-absolutely continuous \(\R^d\)-valued functions, the Stieltjes fundamental theorem gives
\[
f(t)=f(a)+\int_{[a,t)} f'_g(s)\,d\mu_g(s).
\tag{2.13}
\]

\begin{example}[A nonconstant super-\(g\)-H\"older function]
\label{ex:basic-super-holder}
Fix an interior point \(\tau\in I\), let \(\rho>0\), let \(\alpha>1\), and choose \(v\in\R^d\) with \(\|v\|=1\). Define \(g:I\to\R\) and \(f:I\to\R^d\) by
\[
g(t)=\begin{cases}
t,&t\le\tau,\\
t+\rho,&t>\tau,
\end{cases}
\qquad
f(t)=\begin{cases}
0,&t\le\tau,\\
\rho^\alpha v,&t>\tau.
\end{cases}
\tag{2.14}
\]
Then \(g\) is a left-continuous nondecreasing Stieltjes derivator with \(\Delta^+g(\tau)=\rho\), while \(f\) is nonconstant. If \(s\le\tau<t\), then \(\|f(t)-f(s)\|=\rho^\alpha\) and \(g(t)-g(s)=t+\rho-s\ge\rho\), hence \(\|f(t)-f(s)\|\le |g(t)-g(s)|^\alpha\) for all \(s,t\in I\), so \([f]_{\alpha,g}\le1\). Taking \(s=\tau\) and letting \(t\downarrow\tau\) gives ratios converging to \(1\), and therefore \([f]_{\alpha,g}=1\).
\end{example}

Example~\ref{ex:basic-super-holder} shows why the classical constancy argument fails in the presence of a clock jump. Continuous clock increments can be subdivided into arbitrarily small pieces, whereas a positive jump cannot be subdivided in the Stieltjes scale. Section~3 shows that, for \(\alpha>1\), these jumps are the only source of nonconstant behavior.

\section{Atomic representation and function-space structure of super-\(g\)-H\"older functions}

We first show that the nonatomic part of the Stieltjes clock cannot contribute to a super-\(g\)-H\"older function. This leads to a pure-jump representation and an exact formula for the \(g\)-H\"older seminorm. We then use the resulting jump sequence to describe the function space and its approximation properties.

\begin{proposition}
\label{prop:holder-implies-gac}
Let \(f:I\to\R^d\) be \(g\)-H\"older continuous with exponent \(\alpha\ge1\). Then \(f\) is \(g\)-absolutely continuous.
\end{proposition}

\begin{proof}
Let \([f]_{\alpha,g}\le L\). For a finite family of pairwise disjoint intervals \((s_k,t_k)\subset I\), put \(u_k=g(t_k)-g(s_k)\ge0\). If \(\alpha=1\), then
\[
\sum_k\|f(t_k)-f(s_k)\|\le L\sum_k u_k.
\]
Choosing \(\delta>0\) so that \(L\delta<\varepsilon\) proves the claim in this case. If \(\alpha>1\) and \(\sum_k u_k<\delta\), then every \(u_k<\delta\) and
\[
\sum_k\|f(t_k)-f(s_k)\|
\le L\sum_k u_k^\alpha
\le L\delta^{\alpha-1}\sum_k u_k
< L\delta^\alpha.
\]
Choosing \(\delta\) so that \(L\delta^\alpha<\varepsilon\) proves the claim.
\end{proof}

\begin{proposition}
\label{prop:continuous-extinction}
Let \(f:I\to\R^d\) be super-\(g\)-H\"older with exponent \(\alpha>1\). Then (3.1) holds for \(\mu_g^c\)-almost every \(t\in I\):
\[
f'_g(t)=0.
\tag{3.1}
\]
\end{proposition}

\begin{proof}
By Proposition~\ref{prop:holder-implies-gac}, \(f\) is \(g\)-absolutely continuous, and hence \(f'_g\) exists \(\mu_g\)-almost everywhere. Since \(D_g\) is at most countable and \(\mu_g^c\) is nonatomic, \(\mu_g^c(D_g)=0\). Moreover, \(\mu_g(C_g)=0\). Hence, for \(\mu_g^c\)-almost every point \(t\) at which the derivative exists, \(t\notin D_g\cup C_g\); since \(g\) is left-continuous and \(t\notin D_g\), \(g\) is continuous at \(t\). For such a point,
\[
\frac{\|f(t)-f(s)\|}{|g(t)-g(s)|}
\le [f]_{\alpha,g}|g(t)-g(s)|^{\alpha-1}
\]
whenever \(s\to t\) and \(g(s)\ne g(t)\). Since \(\alpha-1>0\) and \(|g(s)-g(t)|\to0\), the right-hand side tends to zero. Hence \(f'_g=0\) \(\mu_g^c\)-almost everywhere.
\end{proof}

\begin{theorem}[Pure-jump representation and exact seminorm formula]
\label{thm:pure-jump-characterization}
Let \(g:I\to\R\) be left-continuous and nondecreasing and let \(\alpha>1\). For \(f:I\to\R^d\), the following are equivalent:
\begin{enumerate}[label=\textup{(\roman*)}]
\item \(f\) is super-\(g\)-H\"older with exponent \(\alpha\);
\item there exists a family \((v_\tau)_{\tau\in D_g}\subset\R^d\) such that, for every \(t\in I\),
\[
f(t)=f(a)+\sum_{\substack{\tau\in D_g\\a\le\tau<t}}v_\tau,
\tag{3.2}
\]
and
\[
\sup_{\tau\in D_g}\frac{\|v_\tau\|}{[\Delta^+g(\tau)]^\alpha}<\infty.
\tag{3.3}
\]
\end{enumerate}
In this representation \(v_\tau=\Delta^+f(\tau)\), and
\[
[f]_{\alpha,g}
=
\sup_{\tau\in D_g}
\frac{\|\Delta^+f(\tau)\|}{[\Delta^+g(\tau)]^\alpha}.
\tag{3.4}
\]
When \(D_g=\varnothing\), the supremum in (3.3)--(3.4) is understood to be zero.
\end{theorem}

\begin{proof}
Assume first that \(f\) is super-\(g\)-H\"older. By Proposition~\ref{prop:holder-implies-gac},
\[
f(t)-f(a)=\int_{[a,t)}f'_g\,d\mu_g.
\]
Using the decomposition \(\mu_g=\mu_g^c+\mu_g^a\), Proposition~\ref{prop:continuous-extinction}, and the atomic derivative formula (2.12),
\[
f(t)-f(a)
=
\sum_{\substack{\tau\in D_g\\\tau<t}}f'_g(\tau)\Delta^+g(\tau)
=
\sum_{\substack{\tau\in D_g\\\tau<t}}\Delta^+f(\tau).
\]
Fix \(\tau\in D_g\). Since \(g(t)\to g(\tau+)\) as \(t\downarrow\tau\), the \(g\)-H\"older estimate shows that \(f(t)\) is Cauchy as \(t\downarrow\tau\). Hence the right limit \(f(\tau+)\) exists. Letting \(t\downarrow\tau\) in the \(g\)-H\"older estimate then gives
\[
\|\Delta^+f(\tau)\|\le [f]_{\alpha,g}[\Delta^+g(\tau)]^\alpha.
\tag{3.5}
\]

Conversely, suppose (3.2)--(3.3) hold and let
\[
C:=\sup_{\tau\in D_g}\frac{\|v_\tau\|}{[\Delta^+g(\tau)]^\alpha}.
\]
Since
\[
\sum_{\tau\in D_g}\Delta^+g(\tau)\le g(b)-g(a)<\infty,
\]
and \(\alpha>1\),
\[
\sum_{\tau\in D_g}[\Delta^+g(\tau)]^\alpha
\le
\left(\sum_{\tau\in D_g}\Delta^+g(\tau)\right)^\alpha<\infty.
\]
Thus the series in (3.2) is absolutely convergent. By (3.2) and the absolute convergence just established, for \(s,t\in I\) with \(s<t\),
\[
\begin{aligned}
\|f(t)-f(s)\|
&\le C\sum_{\substack{\tau\in D_g\\s\le\tau<t}}[\Delta^+g(\tau)]^\alpha\\
&\le C\left(\sum_{\substack{\tau\in D_g\\s\le\tau<t}}\Delta^+g(\tau)\right)^\alpha\\
&\le C|g(t)-g(s)|^\alpha.
\end{aligned}
\]
Hence \([f]_{\alpha,g}\le C\). For each \(\tau\in D_g\), absolute convergence gives \(f(t)\to f(\tau)+v_\tau\) as \(t\downarrow\tau\), while monotonicity of \(g\) gives \(g(t)\to g(\tau+)\). Therefore,
\[
\frac{\|\Delta^+f(\tau)\|}{[\Delta^+g(\tau)]^\alpha}
\le [f]_{\alpha,g},
\]
which gives the reverse inequality in (3.4).
\end{proof}

\subsection{Sequence-space representation}

The pure-jump characterization now allows the super-$g$-H\"older space to be described directly by its normalized jump sequence. We use this representation to determine the basic function-space properties.

For \(f\in\mathcal H_{\alpha,g}(I;\R^d)\) and \(\tau\in D_g\), define
\[
z_\tau(f)
:=
\frac{\Delta^+f(\tau)}{[\Delta^+g(\tau)]^\alpha},
\]
and write
\[
Z_{\alpha,g}(f)
:=
\bigl(z_\tau(f)\bigr)_{\tau\in D_g}.
\]
By Theorem~\ref{thm:pure-jump-characterization},
\[
Z_{\alpha,g}(f)\in\ell^\infty(D_g;\R^d),
\qquad
\|Z_{\alpha,g}(f)\|_{\ell^\infty}
=
[f]_{\alpha,g}.
\]

\begin{theorem}[Sequence-space representation]
\label{thm:jump-space-isometry}
The map
\[
\mathcal J_{\alpha,g}:
\mathcal H_{\alpha,g}(I;\R^d)
\longrightarrow
\R^d\oplus_1\ell^\infty(D_g;\R^d),
\]
defined by
\[
\mathcal J_{\alpha,g}(f)
:=
\bigl(f(a),Z_{\alpha,g}(f)\bigr),
\]
is a linear isometric isomorphism. Its restriction to \(\mathcal H_{\alpha,g}^{0}(I;\R^d)\) gives the linear isometric isomorphism
\[
T_{\alpha,g}:
\mathcal H_{\alpha,g}^{0}(I;\R^d)
\longrightarrow
\ell^\infty(D_g;\R^d),
\qquad
T_{\alpha,g}(f):=Z_{\alpha,g}(f).
\]
\end{theorem}

\begin{proof}
By Theorem~\ref{thm:pure-jump-characterization}, for every \(f\in\mathcal H_{\alpha,g}(I;\R^d)\), one has \([f]_{\alpha,g}=\sup_{\tau\in D_g}\|z_\tau(f)\|=\|Z_{\alpha,g}(f)\|_{\ell^\infty}\), and \(f\) is uniquely determined by \(f(a)\) and its jumps. Hence \(\mathcal J_{\alpha,g}\) is linear and injective. Moreover,
\[
\|\mathcal J_{\alpha,g}(f)\|
=
\|f(a)\|+\|Z_{\alpha,g}(f)\|_{\ell^\infty}
=
\|f(a)\|+[f]_{\alpha,g}
=
\|f\|_{\alpha,g},
\]
so \(\mathcal J_{\alpha,g}\) is an isometry.

It remains to prove surjectivity. Let \((x,\xi)\in\R^d\oplus_1\ell^\infty(D_g;\R^d)\) and define \(f:I\to\R^d\) by
\[
f(t)
:=
x+
\sum_{\substack{\tau\in D_g\\ \tau<t}}
\xi(\tau)[\Delta^+g(\tau)]^\alpha,
\qquad t\in I.
\]
Since \(\xi\in\ell^\infty(D_g;\R^d)\),
\[
\sum_{\tau\in D_g}
\|\xi(\tau)\|[\Delta^+g(\tau)]^\alpha
\le
\|\xi\|_{\ell^\infty}
\sum_{\tau\in D_g}[\Delta^+g(\tau)]^\alpha
\le
\|\xi\|_{\ell^\infty}[g(b)-g(a)]^\alpha
<\infty.
\]
Thus, for every \(t\in I\), the series defining \(f(t)\) converges absolutely, so \(f:I\to\R^d\) is well defined. Applying the converse part of Theorem~\ref{thm:pure-jump-characterization} with \(v_\tau:=\xi(\tau)[\Delta^+g(\tau)]^\alpha\) for \(\tau\in D_g\) gives \(f\in\mathcal H_{\alpha,g}(I;\R^d)\), with \(f(a)=x\) and \(z_\tau(f)=\xi(\tau)\) for every \(\tau\in D_g\). Hence \(Z_{\alpha,g}(f)=\xi\) and \(\mathcal J_{\alpha,g}(f)=(x,\xi)\), so \(\mathcal J_{\alpha,g}\) is surjective. Restricting to \(f\in\mathcal H_{\alpha,g}^{0}(I;\R^d)\), for which \(f(a)=0\), yields the stated linear isometric isomorphism \(T_{\alpha,g}\).
\end{proof}

\begin{corollary}[Banach structure, dimension, and separability]
\label{cor:function-space-structure}
Both \(\mathcal H_{\alpha,g}^{0}(I;\R^d)\) and \(\mathcal H_{\alpha,g}(I;\R^d)\) are Banach spaces. They are separable if and only if \(D_g\) is finite. Suppose that \(D_g\) contains exactly \(N\) jump points. Equivalently, \(\mu_g\) has exactly \(N\) atoms. Then
\[
\dim \mathcal H_{\alpha,g}^{0}(I;\R^d)=dN,
\qquad
\dim \mathcal H_{\alpha,g}(I;\R^d)=d(N+1).
\]
If \(D_g\) is infinite, then both spaces are nonseparable.
\end{corollary}

\begin{proof}
By Theorem~\ref{thm:jump-space-isometry}, \(\mathcal H_{\alpha,g}^{0}(I;\R^d)\) and \(\mathcal H_{\alpha,g}(I;\R^d)\) are isometrically isomorphic to \(\ell^\infty(D_g;\R^d)\) and \(\R^d\oplus_1\ell^\infty(D_g;\R^d)\), respectively. Since these target spaces are Banach, so are the two super-\(g\)-H\"older spaces.

If \(\operatorname{card}(D_g)=N<\infty\), then \(\ell^\infty(D_g;\R^d)\cong(\R^d)^N\). The stated dimension formulas follow, and both spaces are separable. Suppose now that \(D_g\) is infinite. Since \(g\) is nondecreasing, \(D_g\) is countably infinite. Choose \(e\in\R^d\) with \(\|e\|=1\). For each \(A\subset D_g\), define \(\xi_A:D_g\to\R^d\) by \(\xi_A(\tau)=e\) if \(\tau\in A\) and \(\xi_A(\tau)=0\) otherwise. Then \(\xi_A\in\ell^\infty(D_g;\R^d)\), and \(\|\xi_A-\xi_B\|_{\ell^\infty}=1\) whenever \(A\ne B\). Since a countably infinite set has uncountably many subsets, \(\ell^\infty(D_g;\R^d)\) contains an uncountable \(1\)-separated set and is therefore nonseparable. The isometric identifications in Theorem~\ref{thm:jump-space-isometry} transfer this property to both super-\(g\)-H\"older spaces. Hence they are separable if and only if \(D_g\) is finite.
\end{proof}

\paragraph{Remark.}
If \(g\) is continuous, then \(D_g=\varnothing\), and Theorem~\ref{thm:pure-jump-characterization} gives \(\mathcal H_{\alpha,g}^{0}(I;\R^d)=\{0\}\), while \(\mathcal H_{\alpha,g}(I;\R^d)\) consists only of constant functions. Thus, under super-\(g\)-H\"older regularity, the dimension and separability of the function space are determined entirely by the jump structure of the clock.

\paragraph{Remark.}
Theorem~\ref{thm:pure-jump-characterization} shows that every function in \(\mathcal H_{\alpha,g}(I;\R^d)\) is entirely determined by its jumps at the jump points of \(g\). Thus its total variation is governed by the atomic mass distribution of the clock. We therefore set
\[
A_\alpha(g):=\sum_{\tau\in D_g}m_\tau^\alpha.
\]
The next result shows that \(A_\alpha(g)\) is exactly the optimal constant controlling total variation by the super-\(g\)-H\"older seminorm.

\begin{corollary}[Exact variation formula and sharp bounded-variation estimate]
\label{cor:sharp-bv-embedding}
For every \(f\in\mathcal H_{\alpha,g}(I;\R^d)\),
\[
\operatorname{Var}_I(f)
=
\sum_{\tau\in D_g}\|\Delta^+f(\tau)\|
\le
A_\alpha(g)[f]_{\alpha,g}.
\]
Moreover, \(A_\alpha(g)\le [g(b)-g(a)]^\alpha\), and the constant \(A_\alpha(g)\) is sharp:
\[
\sup_{\substack{f\in\mathcal H_{\alpha,g}^{0}(I;\R^d)\\ [f]_{\alpha,g}\le1}}
\operatorname{Var}_I(f)
=
A_\alpha(g).
\]
\end{corollary}

\begin{proof}
By Theorem~\ref{thm:pure-jump-characterization}, \(f\) has an absolutely convergent pure-jump representation. For any partition \(a=t_0<\cdots<t_n=b\), each increment \(f(t_j)-f(t_{j-1})\) is the sum of the jumps with \(t_{j-1}\le\tau<t_j\). The triangle inequality therefore gives
\[
\operatorname{Var}_I(f)
\le
\sum_{\tau\in D_g}\|\Delta^+f(\tau)\|.
\]
Conversely, let \(\varnothing\ne F\subset D_g\) be finite and let \(\varepsilon>0\). By absolute convergence, for each \(\tau\in F\) one can choose \(r_\tau>\tau\), sufficiently close to \(\tau\) and with the corresponding intervals pairwise disjoint, so that the total magnitude of the jumps strictly between \(\tau\) and \(r_\tau\) is less than \(\varepsilon/\operatorname{card}(F)\). A partition containing all \(\tau\) and \(r_\tau\) then gives \(\operatorname{Var}_I(f)\ge\sum_{\tau\in F}\|\Delta^+f(\tau)\|-\varepsilon\). Letting \(\varepsilon\downarrow0\) and then taking the supremum over finite \(F\subset D_g\) yields the reverse inequality.

By the definition of \(z_\tau(f)\), for each \(\tau\in D_g\) we have \(\|\Delta^+f(\tau)\|=m_\tau^\alpha\|z_\tau(f)\|\le m_\tau^\alpha[f]_{\alpha,g}\). Summation gives the stated variation estimate. Moreover,
\[
A_\alpha(g)
=
\sum_{\tau\in D_g}m_\tau^\alpha
\le
\left(\sum_{\tau\in D_g}m_\tau\right)^\alpha
\le
[g(b)-g(a)]^\alpha.
\]
For sharpness, both \([f]_{\alpha,g}\) and \(\operatorname{Var}_I(f)\) are invariant under addition of constants, so the extremal problem may be normalized by \(f(a)=0\). If \(D_g=\varnothing\), then \(A_\alpha(g)=0\) and the assertion is immediate. Otherwise, choose \(v\in\R^d\) with \(\|v\|=1\) and use the converse part of Theorem~\ref{thm:pure-jump-characterization} to prescribe \(\Delta^+f(\tau)=m_\tau^\alpha v\) for \(\tau\in D_g\). Then \(z_\tau(f)=v\) for every \(\tau\in D_g\), so \([f]_{\alpha,g}=1\), while the variation identity above gives \(\operatorname{Var}_I(f)=A_\alpha(g)\). Thus the constant is sharp.
\end{proof}

\subsection{Finite-jump approximation, norm closure, and compact embeddings}

The atomic representation suggests a natural approximation problem. Given \(f\in\mathcal H_{\alpha,g}(I;\R^d)\), we approximate \(f\) by functions \(h\in\mathcal H_{\alpha,g}(I;\R^d)\) that retain nonzero jumps at only finitely many jump points of the clock. More specifically, for a prescribed \(N\in\mathbb N_0\), we restrict the approximating function to have nonzero jumps at at most \(N\) jump points, while imposing \(h(a)=f(a)\) so that the approximation concerns only the atomic part of the function.

For approximation by functions with at most \(N\) nonzero jumps, the primary error is measured in the super-\(g\)-H\"older seminorm,
\[
[f-h]_{\alpha,g}
=
\sup_{\tau\in D_g}
\frac{\|\Delta^+f(\tau)-\Delta^+h(\tau)\|}{m_\tau^\alpha}.
\]
Thus the problem is to retain at most \(N\) nonzero jumps so that the largest remaining normalized jump discrepancy is as small as possible. The results below determine this best approximation error, characterize norm approximation by finite-jump functions, and quantify the improvement obtained when the function has a larger super-\(g\)-H\"older exponent.

For \(f\in\mathcal H_{\alpha,g}(I;\R^d)\), put
\[
q_f(\tau):=
\frac{\|\Delta^+f(\tau)\|}{m_\tau^\alpha},
\qquad \tau\in D_g,
\]
so that \(q_f(\tau)\) is the normalized jump magnitude at \(\tau\). For \(N\in\mathbb N_0\), define the threshold
\[
q^*_{N+1}(f)
:=
\inf\left\{
 r\ge0:
 \operatorname{card}\{\tau\in D_g:q_f(\tau)>r\}\le N
\right\}.
\tag{3.6}
\]
We also define the best error among approximants with at most \(N\) nonzero jumps by
\[
\mathcal E_N^{(\alpha)}(f)
:=
\inf\left\{
 [f-h]_{\alpha,g}:
 h\in\mathcal H_{\alpha,g}(I;\R^d),\ h(a)=f(a),\
 \operatorname{card}\{\tau\in D_g:\Delta^+h(\tau)\ne0\}\le N
\right\}.
\tag{3.7}
\]

\begin{theorem}[Best approximation by at most \(N\) jumps]
\label{thm:best-N-jump}
Let \(\alpha>1\) and \(N\in\mathbb N_0\).
\begin{enumerate}[label=\textup{(\roman*)}]
\item For every \(f\in\mathcal H_{\alpha,g}(I;\R^d)\),
\[
\boxed{\mathcal E_N^{(\alpha)}(f)=q^*_{N+1}(f).}
\tag{3.8}
\]
Thus the best error is exactly the threshold defined in (3.6).

\item Let \(E\subset D_g\) be finite and define
\[
(P_Ef)(t)
:=
f(a)+
\sum_{\substack{\tau\in E\\ \tau<t}}
\Delta^+f(\tau).
\]
Then
\[
[f-P_Ef]_{\alpha,g}
=
\sup_{\tau\in D_g\setminus E}q_f(\tau),
\tag{3.9}
\]
and
\[
\|f-P_Ef\|_\infty
\le
\sum_{\tau\in D_g\setminus E}\|\Delta^+f(\tau)\|
\le
[f]_{\alpha,g}
\sum_{\tau\in D_g\setminus E}m_\tau^\alpha.
\tag{3.10}
\]
Consequently, finite-jump functions are uniformly dense in \(\mathcal H_{\alpha,g}(I;\R^d)\). They are dense in the norm \(\|\cdot\|_{\alpha,g}\) at a given \(f\) if and only if, for every \(\varepsilon>0\),
\[
\operatorname{card}\{\tau\in D_g:q_f(\tau)\ge\varepsilon\}<\infty.
\tag{3.11}
\]
Equivalently, under Theorem~\ref{thm:jump-space-isometry}, their norm closure is \(\R^d\oplus_1 c_0(D_g;\R^d)\).

\item Let \(\beta>\alpha>1\). When \(D_g\) is infinite, enumerate its jump points as \((\tau_k^*)_{k\ge1}\) so that the atom masses \(m_k^*:=\Delta^+g(\tau_k^*)\) satisfy \(m_1^*\ge m_2^*\ge\cdots>0\). When \(D_g\) is finite, use the same ordering and set \(m_k^*=0\) for \(k>\operatorname{card}(D_g)\). Then
\[
\boxed{
\sup_{\substack{
f\in\mathcal H_{\beta,g}^{0}(I;\R^d)\\
[f]_{\beta,g}\le1}}
\mathcal E_N^{(\alpha)}(f)
=
(m_{N+1}^*)^{\beta-\alpha}.
}
\tag{3.12}
\]
Writing \(M_g:=g(b)-g(a)\), one also has the universal rate
\[
(m_{N+1}^*)^{\beta-\alpha}
\le
\left(\frac{M_g}{N+1}\right)^{\beta-\alpha}.
\tag{3.13}
\]
\end{enumerate}
\end{theorem}

\begin{proof}
By Theorem~\ref{thm:jump-space-isometry}, approximation in \([\cdot]_{\alpha,g}\) is equivalent to approximation of the normalized jump sequence in the \(\ell^\infty\) norm.

For (i), the threshold definition in (3.6) is equivalent to
\[
q^*_{N+1}(f)
=
\inf_{\substack{E\subset D_g\\ \operatorname{card}(E)\le N}}
\sup_{\tau\in D_g\setminus E}q_f(\tau).
\tag{3.14}
\]
Indeed, for such a set \(E\), put \(r_E:=\sup_{\tau\in D_g\setminus E}q_f(\tau)\). Then every \(\tau\in D_g\) with \(q_f(\tau)>r_E\) belongs to \(E\), so \(q^*_{N+1}(f)\le r_E\). Conversely, set \(E_*:=\{\tau\in D_g:q_f(\tau)>q^*_{N+1}(f)\}\). Suppose \(\operatorname{card}(E_*)>N\). Choose distinct \(\tau_1,\ldots,\tau_{N+1}\in E_*\). Since \(q_f(\tau_j)>q^*_{N+1}(f)\) for every \(j\), there exists \(\delta>0\) such that \(q_f(\tau_j)>q^*_{N+1}(f)+\delta\) for all \(j\). By the definition of the infimum in (3.6), there exists \(r<q^*_{N+1}(f)+\delta\) with \(\operatorname{card}\{\tau:q_f(\tau)>r\}\le N\). But \(\tau_1,\ldots,\tau_{N+1}\) all belong to this set, a contradiction. Hence \(\operatorname{card}(E_*)\le N\), while \(\sup_{\tau\in D_g\setminus E_*}q_f(\tau)\le q^*_{N+1}(f)\), proving (3.14).

Let \(h\in\mathcal H_{\alpha,g}(I;\R^d)\) have at most \(N\) nonzero jumps and satisfy \(h(a)=f(a)\), and let \(E\subset D_g\) be its jump set. At each \(\tau\in D_g\setminus E\), one has \(\Delta^+h(\tau)=0\); hence Theorem~\ref{thm:pure-jump-characterization} gives \([f-h]_{\alpha,g}\ge\sup_{\tau\in D_g\setminus E}q_f(\tau)\). Taking the infimum and using (3.14) yields \(\mathcal E_N^{(\alpha)}(f)\ge q^*_{N+1}(f)\). Taking \(h=P_{E_*}f\) gives the reverse inequality and proves (3.8).

For (ii), Theorem~\ref{thm:pure-jump-characterization} shows that \(P_Ef\in\mathcal H_{\alpha,g}(I;\R^d)\) and that the jumps of \(f-P_Ef\) are precisely those of \(f\) on \(D_g\setminus E\). Hence (3.9) follows directly, while the pure-jump representation gives (3.10). Since \(\sum_{\tau\in D_g}\|\Delta^+f(\tau)\|<\infty\), the right-hand side of (3.10) tends to zero along any increasing sequence of finite subsets of \(D_g\) whose union is \(D_g\). Hence finite-jump functions are uniformly dense. The lower bound used in part (i), together with (3.9), shows that approximation in \(\|\cdot\|_{\alpha,g}\) is possible exactly under condition (3.11).

For (iii), let \(f\in\mathcal H_{\beta,g}^{0}(I;\R^d)\) satisfy \([f]_{\beta,g}\le1\). For each \(\tau\in D_g\), Theorem~\ref{thm:pure-jump-characterization} gives \(q_f(\tau)\le m_\tau^{\beta-\alpha}\); in particular \(f\in\mathcal H_{\alpha,g}^{0}(I;\R^d)\). Retaining the \(N\) jump points with largest atom masses and applying part (i) yields \(\mathcal E_N^{(\alpha)}(f)\le(m_{N+1}^*)^{\beta-\alpha}\).

If \(D_g\) has at most \(N\) jump points, then \(m_{N+1}^*=0\) and the equality in (3.12) is immediate. Otherwise, choose \(v\in\R^d\) with \(\|v\|=1\) and let \(f_N\in\mathcal H_{\beta,g}^{0}(I;\R^d)\) have jumps \(\Delta^+f_N(\tau_k^*)=(m_k^*)^\beta v\) for \(k=1,\ldots,N+1\), and zero jumps elsewhere. Theorem~\ref{thm:pure-jump-characterization} gives \([f_N]_{\beta,g}=1\). Every approximant with at most \(N\) nonzero jumps leaves at least one of these \(N+1\) jumps unmatched, and therefore \(\mathcal E_N^{(\alpha)}(f_N)\ge(m_{N+1}^*)^{\beta-\alpha}\). This proves (3.12). Finally, since \(\sum_km_k^*\le M_g\), we have \((N+1)m_{N+1}^*\le M_g\), which gives (3.13).
\end{proof}

\paragraph{Remark.}
The two approximation statements in Theorem~\ref{thm:best-N-jump} are genuinely different. Every super-\(g\)-H\"older function can be approximated uniformly by finite-jump functions because its jump amplitudes are absolutely summable. Approximation in the natural super-\(g\)-H\"older norm is more restrictive: it holds exactly when the normalized jump sequence vanishes at infinity, that is, when
\[
\left(\frac{\Delta^+f(\tau)}{m_\tau^\alpha}\right)_{\tau\in D_g}
\in c_0(D_g;\R^d).
\]

\begin{corollary}[Compact embedding between H\"older exponents and finite-jump closure]
\label{cor:compact-embedding-finite-jump-closure}
Let \(1<\alpha<\beta\). The natural inclusion
\[
J_{\beta,\alpha}:\mathcal H_{\beta,g}^{0}(I;\R^d)
\longrightarrow
\mathcal H_{\alpha,g}^{0}(I;\R^d),
\qquad
J_{\beta,\alpha}f=f,
\]
is compact; equivalently, every bounded sequence in \(\mathcal H_{\beta,g}^{0}(I;\R^d)\) has a subsequence converging in \(\mathcal H_{\alpha,g}^{0}(I;\R^d)\).

Let \(\mathcal F_{\alpha,g}^{0}(I;\R^d)\) be the set of all \(f:I\to\R^d\) in \(\mathcal H_{\alpha,g}^{0}(I;\R^d)\) whose jump set \(\{\tau\in D_g:\Delta^+f(\tau)\ne0\}\) is finite, and set
\[
\mathfrak h_{\alpha,g}^{0}(I;\R^d)
:=
\overline{\mathcal F_{\alpha,g}^{0}(I;\R^d)}^{\,\mathcal H_{\alpha,g}^{0}}.
\]
Then
\[
T_{\alpha,g}\bigl(\mathfrak h_{\alpha,g}^{0}(I;\R^d)\bigr)
=
c_0(D_g;\R^d),
\]
Moreover,
\[
\overline{J_{\beta,\alpha}\bigl(\mathcal H_{\beta,g}^{0}(I;\R^d)\bigr)}^{\,\mathcal H_{\alpha,g}^{0}}
=
\mathfrak h_{\alpha,g}^{0}(I;\R^d).
\]
\end{corollary}

\begin{proof}
Under the sequence-space isometries of Theorem~\ref{thm:jump-space-isometry}, the inclusion \(J_{\beta,\alpha}\) becomes the diagonal operator determined by
\[
T_{\alpha,g}(J_{\beta,\alpha}f)(\tau)
=
m_\tau^{\beta-\alpha}T_{\beta,g}(f)(\tau),
\qquad \tau\in D_g.
\]
Order the jump points so that the corresponding masses satisfy \(m_1^*\ge m_2^*\ge\cdots\), and let \(J_{\beta,\alpha}^{(N)}\) denote the truncation to the first \(N\) entries of the jump sequence. Then
\[
\left\|J_{\beta,\alpha}-J_{\beta,\alpha}^{(N)}\right\|
=
(m_{N+1}^*)^{\beta-\alpha}\longrightarrow0.
\]
Thus the compact embedding is the operator-norm limit of finite-rank atomic truncations.

Under \(T_{\alpha,g}\), finite-jump functions correspond exactly to finitely supported elements of \(\ell^\infty(D_g;\R^d)\). Their closure is \(c_0(D_g;\R^d)\), which proves the first identity.

If \(f\in\mathcal H_{\beta,g}^{0}(I;\R^d)\), then its image in the \(\alpha\)-space is \(m_\tau^{\beta-\alpha}T_{\beta,g}(f)(\tau)\), which belong to \(c_0(D_g;\R^d)\). Conversely, every finite-jump function in \(\mathcal H_{\alpha,g}^{0}(I;\R^d)\) also belongs to \(\mathcal H_{\beta,g}^{0}(I;\R^d)\). Therefore the closure of \(J_{\beta,\alpha}(\mathcal H_{\beta,g}^{0}(I;\R^d))\) in \(\mathcal H_{\alpha,g}^{0}(I;\R^d)\) is exactly \(\mathfrak h_{\alpha,g}^{0}(I;\R^d)\).
\end{proof}

\begin{example}[An infinite atomic super-\(g\)-H\"older function]
\label{ex:infinite-atomic-super-holder}
Let \(I=[0,1]\), let \(d\in\mathbb N\), and choose \(v\in\R^d\) with \(\|v\|=1\). For \(k\in\mathbb N\), set \(\tau_k:=1-2^{-k}\), and define the left-continuous nondecreasing Stieltjes derivator \(g:I\to\R\) by
\[
g(t):=\sum_{\tau_k<t}2^{-k}.
\]
Then \(D_g=\{\tau_k:k\in\mathbb N\}\), \(\Delta^+g(\tau_k)=2^{-k}\), and the atoms accumulate at \(1\), while \(\sum_{k\ge1}\Delta^+g(\tau_k)=1\).

Fix \(\alpha>1\) and define \(f:I\to\R^d\) by
\[
f(t):=\sum_{\tau_k<t}2^{-k\alpha}v.
\]
Since \(\Delta^+f(\tau_k)=2^{-k\alpha}v=[\Delta^+g(\tau_k)]^\alpha v\) for every \(k\in\mathbb N\), Theorem~\ref{thm:pure-jump-characterization} gives \([f]_{\alpha,g}=1\). Thus \(f\) is nonconstant and has infinitely many jumps, although \(g(1)-g(0)=1\); moreover, \(f(1)-f(0)=(2^\alpha-1)^{-1}v\).

For \(N\in\mathbb N\), let \(E_N:=\{\tau_1,\ldots,\tau_N\}\) and \(f_N:=P_{E_N}f\). Since all omitted jumps point in the same direction,
\[
\|f-f_N\|_\infty
=
\frac{2^{-N\alpha}}{2^\alpha-1}
\longrightarrow0,
\]
whereas
\[
[f-f_N]_{\alpha,g}=1
\]
for every \(N\in\mathbb N\). Hence finite-jump truncations converge uniformly to \(f\), but not in the super-\(g\)-H\"older norm. In particular, finite Stieltjes mass and super-\(g\)-H\"older regularity alone do not force finitely many nonzero state jumps.
\end{example}

\begin{example}[Higher regularity and finite-jump approximation rate]
\label{ex:higher-regularity-finite-jump-rate}
Keep the clock \(g\) and atoms \((\tau_k)_{k\in\mathbb N}\) from Example~\ref{ex:infinite-atomic-super-holder}. Let \(\beta>\alpha>1\), choose \(v\in\R^d\) with \(\|v\|=1\), and define \(f:I\to\R^d\) by \(f(t):=\sum_{\tau_k<t}2^{-k\beta}v\). Then \(\Delta^+f(\tau_k)=2^{-k\beta}v\), so \([f]_{\beta,g}=1\) and \(q_f(\tau_k)=2^{-k(\beta-\alpha)}\).

For \(E_N:=\{\tau_1,\ldots,\tau_N\}\) and \(f_N:=P_{E_N}f\), Theorem~\ref{thm:best-N-jump} gives
\[
[f-f_N]_{\alpha,g}
=
2^{-(N+1)(\beta-\alpha)}
\longrightarrow0.
\]
Thus the additional exponent \(\beta-\alpha\) exactly determines the convergence rate of finite-jump approximation in the weaker super-\(g\)-H\"older norm.
\end{example}

\section{Atomic structure of super-\(g\)-H\"older Stieltjes dynamics}

A recent complementary approach connects Stieltjes differential equations with ordinary differential equations through a generalized-inverse change of variables \cite{AdrianTojo2025Connection}. Here we keep the Stieltjes clock explicit and use the super-\(g\)-H\"older assumption to identify the continuous and jump contributions to the solution.

Throughout this section, \(g:I\to\R\) denotes the fixed Stieltjes derivator introduced in Section~2; for \(\tau\in D_g\), \(m_\tau=\Delta^+g(\tau)\) denotes the corresponding atomic mass. From this point onward, the state space is \(\R^n\). Let \(F:I\times\R^n\to\R^n\) be a single-valued vector field and consider
\[
x'_g(t)=F(t,x(t)),\qquad x(a)=x_0\in\R^n.
\tag{4.1}
\]
By a solution we mean a \(g\)-absolutely continuous function \(x:I\to\R^n\) satisfying (4.1) \(\mu_g\)-almost everywhere together with the initial condition. No separate existence theorem is applied here. Fix \(L>0\) and \(\alpha>1\). Because \(g\) is left-continuous, \(x(\tau)\) is the state before the jump and \(x(\tau+)\) the state after the jump, so
\[
\Delta^+x(\tau)=x(\tau+)-x(\tau).
\tag{4.2}
\]

\begin{theorem}[Atomic representation of super-\(g\)-H\"older solutions]
\label{thm:atomic-dynamics}
Let \(g:I\to\R\) be left-continuous and nondecreasing, let \(\alpha>1\) and \(L>0\), and let \(F:I\times\R^n\to\R^n\). Suppose \(x:I\to\R^n\) satisfies \(x(a)=x_0\). The following are equivalent.
\begin{enumerate}[label=\textup{(\roman*)}]
\item \(x\) is a solution of (4.1) and \([x]_{\alpha,g}\le L\).
\item The vector field vanishes along the trajectory on the continuous part of the clock,
\[
F(t,x(t))=0
\qquad\text{for }\mu_g^c\text{-almost every }t\in I,
\tag{4.3}
\]
and, for every \(t\in I\),
\[
x(t)
=
x_0+
\sum_{\substack{\tau\in D_g\\ \tau<t}}
 m_\tau F(\tau,x(\tau)),
\tag{4.4}
\]
and
\[
\sup_{\tau\in D_g}
\frac{\|F(\tau,x(\tau))\|}{m_\tau^{\alpha-1}}
\le L.
\tag{4.5}
\]
\end{enumerate}
Whenever these conditions hold,
\[
[x]_{\alpha,g}
=
\sup_{\tau\in D_g}
\frac{\|F(\tau,x(\tau))\|}{m_\tau^{\alpha-1}},
\]
with the usual convention that the supremum is zero if \(D_g=\varnothing\). In particular, at every jump point \(\tau\in D_g\), equivalently at every atom of \(\mu_g\),
\[
\Delta^+x(\tau)=m_\tau F(\tau,x(\tau)).
\]
Thus, in the super-\(g\)-H\"older regime, the Stieltjes equation is equivalent to vanishing of the vector field along the trajectory on the nonatomic part of the clock together with the jump recurrence (4.4).
\end{theorem}

\begin{proof}
Assume (i). Proposition~\ref{prop:continuous-extinction} gives \(x'_g=0\) for \(\mu_g^c\)-almost every \(t\in I\). Since \(x'_g(t)=F(t,x(t))\) \(\mu_g\)-almost everywhere, (4.3) follows. At each \(\tau\in D_g\), the differential equation also holds because \(\mu_g(\{\tau\})=m_\tau>0\); the atomic derivative formula therefore gives the atomic relation stated above. Theorem~\ref{thm:pure-jump-characterization} then yields the recurrence (4.4) and the exact seminorm identity, from which (4.5) follows.

Conversely, assume (ii). By (4.5), the series in (4.4) is absolutely convergent, since \(\sum_{\tau\in D_g}m_\tau^\alpha<\infty\). For \(s,t\in I\) with \(s<t\),
\[
\|x(t)-x(s)\|
\le
L\sum_{\substack{\tau\in D_g\\s\le\tau<t}}m_\tau^\alpha
\le
L\left(\sum_{\substack{\tau\in D_g\\s\le\tau<t}}m_\tau\right)^\alpha
\le
L|g(t)-g(s)|^\alpha.
\]
Hence \(x\) is super-\(g\)-H\"older, and Proposition~\ref{prop:holder-implies-gac} gives \(g\)-absolute continuity. Absolute convergence of (4.4) gives the atomic relation stated in the theorem, and therefore the differential equation holds at every atom of \(\mu_g\). On the nonatomic part, Proposition~\ref{prop:continuous-extinction} gives \(x'_g=0\) \(\mu_g^c\)-almost everywhere, while (4.3) gives \(F(t,x(t))=0\) there. Thus (4.1) holds \(\mu_g\)-almost everywhere. Finally, Theorem~\ref{thm:pure-jump-characterization} gives the asserted exact seminorm identity.
\end{proof}

For a solution \(x:I\to\R^n\) of (4.1), a clock jump changes the state exactly when \(\Delta^+x(\tau)\ne0\), equivalently when \(F(\tau,x(\tau))\ne0\). If \(F(\tau,x(\tau))=0\), the clock jumps but the state does not.

\begin{corollary}[Lower bounds for nonzero state jumps and their number]
\label{cor:active-lower-bound}
\label{cor:finite-active-capacity}
\label{cor:purely-atomic-clock}
Let \(x:I\to\R^n\) satisfy the equivalent conditions of Theorem~\ref{thm:atomic-dynamics}.
\begin{enumerate}[label=\textup{(\roman*)}]
\item Every nonzero state jump point satisfies the local lower bound
\[
m_\tau
\ge
\left(
\frac{\|F(\tau,x(\tau))\|}{L}
\right)^{1/(\alpha-1)}.
\tag{4.6}
\]
\item Suppose there is \(c>0\), independent of \(\tau\), such that \(\|F(\tau,x(\tau))\|\ge c\) at every nonzero state jump point. Then every nonzero state jump point has mass at least
\[
m_*:=\left(\frac{c}{L}\right)^{1/(\alpha-1)},
\tag{4.7}
\]
and, with
\[
M:=\mu_g([a,b))=g(b)-g(a),
\tag{4.8}
\]
\[
\operatorname{card}\{\tau\in D_g:\Delta^+x(\tau)\ne0\}
\le
\left\lfloor
M\left(\frac{L}{c}\right)^{1/(\alpha-1)}
\right\rfloor.
\tag{4.9}
\]
\item If, more strongly,
\[
\|F(t,x(t))\|\ge c
\qquad\text{for every }t\in I,
\tag{4.10}
\]
then the Stieltjes measure is purely atomic, \(\mu_g^c=0\), every jump point of \(g\) changes the state, and
\[
\operatorname{card}(D_g)
\le
\left\lfloor
M\left(\frac{L}{c}\right)^{1/(\alpha-1)}
\right\rfloor.
\tag{4.11}
\]
\end{enumerate}
\end{corollary}

\begin{proof}
Part (i) follows by rearranging (4.5) at a nonzero state jump point. Under the hypothesis of (ii), (4.6) gives \(m_\tau\ge m_*\) at every nonzero state jump point. Since the sum of all clock jump sizes does not exceed \(M\), the set of nonzero state jump points is finite and
\[
\operatorname{card}\{\tau\in D_g:\Delta^+x(\tau)\ne0\}\,m_*\le M,
\]
which gives (4.9).

Under (4.10), the lower bound \(\|F(t,x(t))\|\ge c>0\) is incompatible with (4.3) on any set of positive \(\mu_g^c\)-measure; hence \(\mu_g^c=0\). Moreover, \(F(\tau,x(\tau))\ne0\) at every \(\tau\in D_g\), so every clock jump changes the state, and (4.11) follows from part (ii).
\end{proof}

Thus a uniform lower bound on the vector field along nonzero state jumps converts the super-\(g\)-H\"older estimate into a positive minimum clock jump for every state change. Since the total Stieltjes mass is finite, the number of nonzero state jumps is then finite.

For later inverse arguments it is useful to record the corresponding zero set along the trajectory
\[
Z_x:=\{t\in I:F(t,x(t))=0\}.
\tag{4.12}
\]
Theorem~\ref{thm:atomic-dynamics} gives \(\mu_g^c(I\setminus Z_x)=0\), and every clock jump with zero state change belongs to \(Z_x\). The nonatomic Stieltjes mass that is not detected by the state is concentrated on \(Z_x\). Thus Stieltjes mass that does not affect the observed state can occur only on this zero set. If the nonzero state jump points are finite, \(\tau_1<\cdots<\tau_N\), then the state does not change between two successive nonzero state jumps: the continuous part produces no motion and the remaining clock jumps have zero state jump. Hence \(x(\tau_j)=x_{j-1}\), and (4.4) reduces to the recurrence
\[
x_j=x_{j-1}+m_jF(\tau_j,x_{j-1}),
\qquad j=1,\ldots,N,
\tag{4.13}
\]
where \(m_j=\Delta^+g(\tau_j)\) and \(x_j\) is the state immediately after the \(j\)-th nonzero state jump.

\section{Thresholds for the number of nonzero state jumps in affine dynamics}

This section specializes Section~4 to affine dynamics and determines how many nonzero state jumps can occur for a prescribed total Stieltjes mass \(M\).

Consider a solution \(x:I\to\R^n\) of
\[
x'_g(t)=\sigma x(t)+\eta,\qquad \sigma>0,\qquad \eta\in\R^n,
\tag{5.1}
\]
with \(x(a)=x_0\in\R^n\), and let \(x_*=-\eta/\sigma\). Assume \(x_0\ne x_*\), \([x]_{\alpha,g}\le L\), \(L>0\), \(\alpha>1\), and prescribe
\[
M=\mu_g([a,b))=g(b)-g(a)>0.
\tag{5.3}
\]

In Sections~5--6 the Stieltjes derivator is not fixed. We consider all left-continuous nondecreasing \(g:I\to\R\) with total Stieltjes mass \(M\) that support a solution satisfying the stated regularity bound. Set \(y(t):=x(t)-x_*\), so that \(y'_g(t)=\sigma y(t)\). If \(\tau_1<\tau_2<\cdots\) are the clock jump points, write \(m_j:=\Delta^+g(\tau_j)\), \(x_j:=x(\tau_j+)\), \(y_j:=x_j-x_*\), and \(R_j:=\|y_j\|\), with \(R_0=\|x_0-x_*\|>0\). Then
\[
R_j=R_0\prod_{k=1}^j(1+\sigma m_k).
\tag{5.8}
\]
Since every factor is larger than one, the affine vector field never vanishes along the trajectory. Corollary~\ref{cor:purely-atomic-clock} therefore gives \(\mu_g^c=0\), and every clock jump changes the state. At the \(j\)-th jump, the super-\(g\)-H\"older bound gives
\[
m_j^{\alpha-1}\ge\frac{\sigma R_0}{L}\prod_{k<j}(1+\sigma m_k).
\tag{5.10}
\]
Define recursively the minimal admissible jump sizes by
\[
\lambda_1:=\left(\frac{\sigma R_0}{L}\right)^{1/(\alpha-1)},
\tag{5.12}
\]
and
\[
\lambda_j:=\left[\frac{\sigma R_0}{L}\prod_{k<j}(1+\sigma\lambda_k)\right]^{1/(\alpha-1)},\qquad j\ge2.
\tag{5.13}
\]

\begin{proposition}[Recurrence for the minimal admissible jump sizes]
\label{prop:event-escalation}
The sequence \((\lambda_j)_{j\in\mathbb N}\) satisfies
\[
\lambda_{j+1}=\lambda_j(1+\sigma\lambda_j)^{1/(\alpha-1)},
\tag{5.14}
\]
and therefore
\[
0<\lambda_1<\lambda_2<\lambda_3<\cdots.
\tag{5.15}
\]
\end{proposition}

\begin{proof}
Dividing the defining identities for \(\lambda_{j+1}^{\alpha-1}\) and \(\lambda_j^{\alpha-1}\) gives \((\lambda_{j+1}/\lambda_j)^{\alpha-1}=1+\sigma\lambda_j\), which is (5.14). Since \(\sigma>0\) and \(\lambda_j>0\), the multiplying factor is larger than one.
\end{proof}

Set \(S_0:=0\) and, for \(N\in\mathbb N\),
\[
S_N:=\sum_{j=1}^N\lambda_j.
\tag{5.17}
\]

\begin{theorem}[Exact threshold for \(N\) nonzero state jumps]
\label{thm:event-birth-thresholds}
Let \(\sigma>0\), \(\eta\in\R^n\), \(x_0\in\R^n\), and \(x_*=-\eta/\sigma\), with \(x_0\ne x_*\). Let \(\alpha>1\), \(L>0\), and \(M>0\). For \(N\in\mathbb N\), there exist a left-continuous nondecreasing \(g:I\to\R\) with \(g(b)-g(a)=M\) and a solution \(x:I\to\R^n\) of (5.1), satisfying \(x(a)=x_0\), \([x]_{\alpha,g}\le L\), and having exactly \(N\) nonzero state jumps if and only if
\[
M\ge S_N.
\tag{5.18}
\]
\end{theorem}

\begin{proof}
For any admissible jump sizes \((m_1,\ldots,m_N)\), inequalities (5.10) imply inductively that \(m_j\ge\lambda_j\) for every \(j\). Hence a solution with exactly \(N\) jumps satisfies \(M=\sum_{j=1}^Nm_j\ge S_N\).

Conversely, assume \(M\ge S_N\). Take \(m_j=\lambda_j\) for \(j=1,\ldots,N-1\) and \(m_N=M-\sum_{j=1}^{N-1}\lambda_j\ge\lambda_N\). Then all inequalities (5.10) hold. Choose \(a\le\tau_1<\cdots<\tau_N<b\), define \(g(t):=g(a)+\sum_{\tau_j<t}m_j\), and generate the states recursively by \(x_j=x_{j-1}+m_j(\sigma x_{j-1}+\eta)\). The resulting step trajectory satisfies the affine equation and the required super-\(g\)-H\"older bound. Since \(x_0\ne x_*\), every clock jump changes the state.
\end{proof}

For \(M>0\), define
\[
N_{\max}(M):=\max\{N\in\mathbb N_0:S_N\le M\}.
\tag{5.19}
\]
This maximum is finite because \(S_N\ge N\lambda_1\).

\begin{corollary}[Exact maximum number of nonzero state jumps]
\label{cor:Nmax}
For \(M>0\), \(N_{\max}(M)=0\) if and only if \(M<S_1\). More generally, for every \(N\in\mathbb N\),
\[
N_{\max}(M)=N\iff S_N\le M<S_{N+1}.
\tag{5.20}
\]
\end{corollary}

\begin{proof}
The sequence \((S_N)\) is strictly increasing, so the assertion follows directly from the definition of \(N_{\max}(M)\).
\end{proof}

\begin{corollary}[An explicit upper bound for \(N_{\max}(M)\)]
Let \(q:=(1+\sigma\lambda_1)^{1/(\alpha-1)}>1\). Then
\[
\lambda_j\ge\lambda_1q^{j-1},\qquad
S_N\ge\lambda_1\frac{q^N-1}{q-1},
\]
and
\[
N_{\max}(M)\le
\left\lfloor
\frac{\log\!\left(1+\dfrac{M(q-1)}{\lambda_1}\right)}{\log q}
\right\rfloor.
\tag{5.22}
\]
\end{corollary}

\begin{proof}
By (5.14) and \(\lambda_j\ge\lambda_1\), one has \(\lambda_{j+1}\ge q\lambda_j\). Summing gives the estimate for \(S_N\), and solving \(S_{N_{\max}(M)}\le M\) gives the result.
\end{proof}

\begin{example}[Maximum number of nonzero state jumps in \(\R^2\)]
\label{ex:event-capacity}
Take \(\sigma=1\), \(\eta=0\), \(x_0=(1,0)\), \(\alpha=2\), \(L=4\), and \(M=1\). Then \(x_*=0\), \(R_0=1\), and \(\lambda_1=0.25\), \(\lambda_2=0.3125\), \(\lambda_3=0.41015625\), and \(\lambda_4\approx0.578384\). Hence \(S_3=0.97265625<1<S_4\approx1.55104\), so \(N_{\max}(1)=3\). The first-jump bound alone gives the cruder estimate \(\lfloor M/\lambda_1\rfloor=4\); the recursive thresholds exclude a fourth jump.
\end{example}

\section{Extremal terminal distances for affine dynamics}

Section~5 determines which numbers of nonzero state jumps are feasible for a prescribed total Stieltjes mass \(M\). Fix such an \(N\). We now optimize the terminal distance over the admissible clock-jump sizes. The lower bound for each jump size depends on the preceding jumps.

For \(N\in\mathbb N\) and \(M>0\), define
\[
\mathcal A_N(M):=\left\{m\in(0,\infty)^N:\sum_{j=1}^Nm_j=M,\quad
m_j^{\alpha-1}\ge\frac{\sigma R_0}{L}\prod_{k<j}(1+\sigma m_k)\right\}.
\tag{6.1}
\]
and
\[
R_N(m):=R_0\prod_{j=1}^N(1+\sigma m_j).
\tag{6.2}
\]

If \(m_j>m_{j+1}\), interchanging these two values preserves admissibility and leaves \(R_N\) unchanged. Hence it is enough to consider
\[
m_1\le m_2\le\cdots\le m_N.
\tag{6.3}
\]

\begin{theorem}[Exact fixed-\(N\) minimum]
\label{thm:fixedN-min}
Let \(N\in\mathbb N\) and \(M\ge S_N\). The unique nondecreasing minimizer is
\[
m_j^{\min}=\lambda_j,\qquad j=1,\ldots,N-1,
\tag{6.4}
\]
and
\[
m_N^{\min}=M-\sum_{j=1}^{N-1}\lambda_j.
\tag{6.5}
\]
Consequently,
\[
R_{N,\min}=R_0\left[\prod_{j=1}^{N-1}(1+\sigma\lambda_j)\right]
\left[1+\sigma\left(M-\sum_{j=1}^{N-1}\lambda_j\right)\right].
\tag{6.6}
\]
\end{theorem}

\begin{proof}
For a nondecreasing admissible sequence, suppose that \(m_1,\ldots,m_{j-1}\) have been fixed at \(\lambda_1,\ldots,\lambda_{j-1}\). The smallest admissible value of \(m_j\) is then \(\lambda_j\). If \(m_j>\lambda_j\) for some \(j<N\), transferring the excess to the last coordinate preserves admissibility and strictly decreases the product. Hence (6.4)--(6.5) are necessary and give the minimum.
\end{proof}

\subsection{Fixed-\(N\) maximum}

For the following argument, write \(p:=\alpha-1\) and \(c:=\sigma R_0/L\). For a nondecreasing admissible sequence,
\[
r_j:=\log(1+\sigma m_j).
\tag{6.7}
\]
Then
\[
\log\frac{R_N}{R_0}=\sum_{j=1}^Nr_j,
\tag{6.8}
\]
and
\[
r_j\ge h\!\left(\sum_{k<j}r_k\right),
\tag{6.9}
\]
where
\[
h(s):=\log\!\left[1+\sigma c^{1/p}e^{s/p}\right].
\tag{6.10}
\]
A constraint is called saturated when equality holds and strict otherwise. The fixed-mass condition is
\[
\sum_{j=1}^Ne^{r_j}=N+\sigma M.
\tag{6.11}
\]

\begin{lemma}[Form of a nondecreasing maximizer]
\label{lem:prefix-tail}
Let \(m\in\mathcal A_N(M)\) be a nondecreasing maximizer. Then there exist \(K\in\{1,\ldots,N\}\) and \(\rho>0\) such that
\[
r_1=\cdots=r_K=\rho,
\tag{6.12}
\]
and, for every \(j>K\),
\[
r_j=h\!\left(\sum_{k<j}r_k\right).
\tag{6.13}
\]
\end{lemma}

\begin{proof}
If two consecutive constraints are strict and \(r_j<r_{j+1}\), a sufficiently small transfer from \(r_{j+1}\) to \(r_j\) preserves the local constraints and the value of \(r_j+r_{j+1}\), while strict convexity reduces \(e^{r_j}+e^{r_{j+1}}\). The saved mass can then be added to the last coordinate, increasing the objective. Hence, whenever two consecutive constraints are strict, the corresponding consecutive coordinates are equal. Once a constraint is saturated, the strict increase of \(h\) and the same exchange argument force every later constraint to be saturated.
\end{proof}

For \(q\in[\lambda_1,\infty)\), define
\[
\widehat m_1(q):=q,
\]
and, for \(j=2,\ldots,N\),
\[
\widehat m_j(q):=\max\left\{q,\left[\frac{\sigma R_0}{L}\prod_{k<j}(1+\sigma \widehat m_k(q))\right]^{1/(\alpha-1)}\right\}.
\tag{6.14}
\]
Set
\[
G_N(q):=\sum_{j=1}^N\widehat m_j(q).
\tag{6.15}
\]

\begin{theorem}[Exact fixed-\(N\) maximum]
\label{thm:fixedN-max}
Let \(N\in\mathbb N\) and \(M\ge S_N\). There is a unique \(q_N(M)\in[\lambda_1,\infty)\) such that
\[
G_N(q_N(M))=M.
\tag{6.16}
\]
The unique nondecreasing maximizer is
\[
m_j^*=\widehat m_j(q_N(M)),\qquad j=1,\ldots,N,
\tag{6.17}
\]
and
\[
R_{N,\max}=R_0\prod_{j=1}^N(1+\sigma m_j^*).
\tag{6.18}
\]
\end{theorem}

\begin{proof}
Every admissible vector satisfies \(m_j\ge\lambda_1>0\), so \(\mathcal A_N(M)\) is a closed subset of the compact simplex \(\{m_j\ge\lambda_1,\ \sum_jm_j=M\}\). Hence a maximum exists. By the ordering observation above it may be taken nondecreasing, and Lemma~\ref{lem:prefix-tail} gives the form (6.14). Each \(\widehat m_j\) is continuous and nondecreasing, \(\widehat m_1(q)=q\), \(G_N(\lambda_1)=S_N\), and \(G_N(q)\to\infty\). Hence \(G_N\) is strictly increasing and (6.16) has a unique solution.
\end{proof}

\begin{corollary}[Criterion for equal jump sizes]
\label{cor:equal-max}
The equal allocation \(m_j=M/N\) is the fixed-\(N\) maximizer if and only if
\[
\left(\frac{M}{N}\right)^{\alpha-1}
\ge
\frac{\sigma R_0}{L}\left(1+\frac{\sigma M}{N}\right)^{N-1}.
\tag{6.19}
\]
\end{corollary}

\begin{proof}
For a fixed sum, \(\prod_{j=1}^N(1+\sigma m_j)\) is maximized by equal jump sizes. Thus the equal allocation is the admissible maximizer exactly when it satisfies the last, and strongest, constraint in (6.1).
\end{proof}

\subsection{Global extrema}

\begin{proposition}[Global minimum and maximum]
\label{thm:global-min}
\label{thm:global-max}
If \(M\ge S_1\), then
\[
R_{\min}^{\rm global}=R_0(1+\sigma M),
\tag{6.20}
\]
attained by a single clock jump of size \(M\), while
\[
R_{\max}^{\rm global}=\max_{1\le N\le N_{\max}(M)}R_{N,\max}.
\tag{6.21}
\]
\end{proposition}

\begin{proof}
For every positive allocation,
\[
\prod_{j=1}^N(1+\sigma m_j)\ge 1+\sigma M,
\]
with strict inequality for \(N\ge2\). The maximum follows by taking the largest fixed-\(N\) maximum over the finitely many feasible values of \(N\).
\end{proof}

After shifting the equilibrium by \(y=x-x_*\), the affine equation is \(y'_g=\sigma y\). If the same total increment \(M\) is used as a continuous clock variable \(s\in[0,M]\), then \(dy/ds=\sigma y\) and the terminal distance is \(R_0e^{\sigma M}\). Since \(\log(1+\sigma m_j)<\sigma m_j\) for every \(m_j>0\), every response corresponding to admissible clock-jump sizes satisfies
\[
R_N(m)<R_0e^{\sigma M},
\]
and hence \(R_{\max}^{\rm global}<R_0e^{\sigma M}\).

\begin{example}[Complete extremal calculation]
\label{ex:complete-extremal}
For the data in Example~\ref{ex:event-capacity}, \(N_{\max}=3\). For one jump, \(R_1=2\). For two jumps, the fixed-\(2\) minimum \((0.25,0.75)\) gives \(R_{2,\min}=2.1875\), while \((0.5,0.5)\) gives \(R_{2,\max}=2.25\). For three jumps, the fixed-\(3\) minimum is \((0.25,0.3125,0.4375)\), with \(R_{3,\min}=2.3583984375\). The fixed-\(3\) maximum is determined by
\[
q=-13+4\sqrt{11}\approx0.266499,
\]
with maximizing jump sizes
\[
m^*=\bigl(-13+4\sqrt{11},\,-3+\sqrt{11},\,17-5\sqrt{11}\bigr).
\]
Thus \(R_{3,\max}=2324-700\sqrt{11}\approx2.362647\), so \(R_{\max}^{\rm global}\approx2.362647\). The continuous-clock value is \(e\approx2.718282\).
\end{example}

\section{Recovery of a Stieltjes clock from the state trajectory}

The inverse problem is to determine which parts of a Stieltjes derivator \(g:I\to\R\) can be recovered from an observed state trajectory. Throughout this section, \(g\) is a Stieltjes derivator as in Section~2. Consider
\[
x'_g(t)=F(t,x(t)),
\tag{7.1}
\]
where \(F:I\times\R^n\to\R^n\) is known and \(x:I\to\R^n\) is a super-\(g\)-H\"older solution with exponent \(\alpha>1\) and \([x]_{\alpha,g}\le L\), \(L>0\), whose jump times and corresponding state jumps are observed. The inverse statements below are conditional on this known vector field and observed trajectory; no separate forward existence theorem is applied. For a jump point \(\tau\in D_g\), retain the notation \(m_\tau=\Delta^+g(\tau)\) from Section~2 and write
\[
z_\tau:=\Delta^+x(\tau).
\tag{7.2}
\]
At every such point,
\[
z_\tau=m_\tau F(\tau,x(\tau)).
\tag{7.3}
\]

\[
D_x:=\{\tau\in[a,b):\Delta^+x(\tau)\ne0\}.
\]

\begin{theorem}[Recovery of a clock jump from an observed state jump]
\label{thm:visible-recovery}
Let \(F:I\times\R^n\to\R^n\) be known, and let \(x:I\to\R^n\) be a super-\(g\)-H\"older solution of (7.1) of exponent \(\alpha>1\), with \([x]_{\alpha,g}\le L\) for some \(L>0\), and suppose that its jump times and state jumps are observed. Fix \(\tau\in D_x\) and set \(z_\tau=\Delta^+x(\tau)\). Then \(\tau\in D_g\), \(F(\tau,x(\tau))\ne0\), and the corresponding clock jump is uniquely determined by
\[
m_\tau=\frac{\|z_\tau\|}{\|F(\tau,x(\tau))\|},
\tag{7.4}
\]
and the consistency condition
\[
z_\tau\in\mathbb R_+F(\tau,x(\tau)),
\tag{7.5}
\]
where \(\mathbb R_+F(\tau,x(\tau)):=\{rF(\tau,x(\tau)):r>0\}\). Equivalently,
\[
m_\tau=
\frac{\langle z_\tau,F(\tau,x(\tau))\rangle}
{\|F(\tau,x(\tau))\|^2}.
\tag{7.6}
\]
\end{theorem}

\begin{proof}
By Theorem~\ref{thm:pure-jump-characterization}, every nonzero jump of \(x\) occurs at a jump point of \(g\), so \(\tau\in D_g\) and \(m_\tau:=\Delta^+g(\tau)>0\). Equation (7.3) then gives \(z_\tau=m_\tau F(\tau,x(\tau))\). Since \(z_\tau\ne0\), the vector field is nonzero and the positive-ray condition follows. Taking norms gives (7.4), while taking the inner product with \(F(\tau,x(\tau))\) gives (7.6).
\end{proof}

\begin{corollary}[Necessary H\"older bound for a recovered clock jump]
For an observed state jump \(\tau\in D_x\) under the assumptions of Theorem~\ref{thm:visible-recovery}, the recovered mass satisfies
\[
m_\tau\ge\left(\frac{\|z_\tau\|}{L}\right)^{1/\alpha},
\tag{7.7}
\]
equivalently
\[
\|F(\tau,x(\tau))\|^\alpha\le L\|z_\tau\|^{\alpha-1}.
\tag{7.8}
\]
\end{corollary}

\begin{proof}
The atomic H\"older estimate \(\|z_\tau\|\le Lm_\tau^\alpha\) gives (7.7). Substituting the recovered mass from (7.4) then gives (7.8).
\end{proof}

Recall the zero set of the vector field along the trajectory,
\[
Z_x=\{t\in I:F(t,x(t))=0\},
\]
introduced in Section~4.

\begin{theorem}[Observed state jumps and clock jumps]
\label{thm:visibility-decomposition}
Let \(F:I\times\R^n\to\R^n\), and let \(x:I\to\R^n\) be a super-\(g\)-H\"older solution of (7.1) with exponent \(\alpha>1\) and \([x]_{\alpha,g}\le L\) for some \(L>0\), and suppose that its jump times and state jumps are observed. Then
\[
D_x=\{\tau\in D_g:F(\tau,x(\tau))\ne0\},
\tag{7.10}
\]
so \(D_g\setminus D_x\subseteq Z_x\), and
\[
\mu_g^c(I\setminus Z_x)=0.
\tag{7.11}
\]
Thus Stieltjes mass that is not detected by the state can occur only where the known vector field vanishes along the trajectory.
\end{theorem}

\begin{proof}
By the pure-jump representation of Section~3, every nonzero jump of \(x\) occurs at a jump point of \(g\); hence \(D_x\subseteq D_g\). For \(\tau\in D_g\), equation (7.3) and \(m_\tau>0\) show that \(\tau\in D_x\) if and only if \(F(\tau,x(\tau))\ne0\). This proves (7.10). Equation (7.11) is Theorem~\ref{thm:atomic-dynamics} applied to the same solution.
\end{proof}

\begin{theorem}[All Stieltjes measures compatible with the observed trajectory]
\label{thm:compatible-single-clocks}
Let \(F:I\times\R^n\to\R^n\) be known and let \(x:I\to\R^n\) be an observed trajectory, including its individual state jumps, that is a super-\(g\)-H\"older solution of (7.1) for some Stieltjes clock. Define the measure recovered from the observed state jumps
\[
\mu_{\rm obs}:=\sum_{\tau\in D_x}m_\tau\delta_\tau,
\qquad
m_\tau:=\frac{\|\Delta^+x(\tau)\|}{\|F(\tau,x(\tau))\|}.
\]
Then every compatible finite Stieltjes measure has the form
\[
\mu_g=\mu_{\rm obs}+\nu,
\]
where \(\nu\) is a finite nonnegative Borel measure concentrated on \(Z_x\cap[a,b)\). Conversely, every such \(\nu\) defines, up to an additive normalization, a left-continuous nondecreasing Stieltjes clock compatible with the same observed trajectory and the same super-\(g\)-H\"older bound.
\end{theorem}

\begin{proof}
Theorem~\ref{thm:visibility-decomposition} and Theorem~\ref{thm:visible-recovery} show that the Stieltjes mass outside \(Z_x\) is exactly the measure recovered from the observed state jumps. Hence every compatible Stieltjes measure has the stated form. Conversely, let \(\widetilde\mu:=\mu_{\rm obs}+\nu\) and for an arbitrary \(C\in\R\), define \(\widetilde g(t):=C+\widetilde\mu([a,t))\). Since \(F(t,x(t))=0\) for \(\nu\)-almost every \(t\),
\[
\int_{[a,t)}F(s,x(s))\,d\nu(s)=0
\]
for every \(t\in I\). The recovered atoms reproduce the observed state jumps, so the Stieltjes integral equation for \(x\) is unchanged. Moreover, for \(s<t\), the atomic seminorm estimate gives
\[
\|x(t)-x(s)\|
\le
L\sum_{\tau\in D_x\cap[s,t)}m_\tau^\alpha
\le
L\left(\sum_{\tau\in D_x\cap[s,t)}m_\tau\right)^\alpha
\le
L[\widetilde\mu([s,t))]^\alpha.
\]
Hence \(x\) is super-\(\widetilde g\)-H\"older with the same bound, and Proposition~\ref{prop:holder-implies-gac} gives \(\widetilde g\)-absolute continuity. Thus \(\widetilde g\) is compatible with the same trajectory.
\end{proof}

\begin{corollary}[Uniqueness of the Stieltjes clock]
\label{thm:complete-single-recovery}
The Stieltjes clock generating the observed trajectory together with its individual state jumps is uniquely determined up to an additive constant if and only if
\[
Z_x\cap[a,b)=\varnothing.
\]
Equivalently, \(F(t,x(t))\ne0\) for every \(t\in[a,b)\).
\end{corollary}

\begin{proof}
If \(Z_x\cap[a,b)=\varnothing\), Theorem~\ref{thm:compatible-single-clocks} forces \(\nu=0\), so the measure recovered from the observed state jumps is the unique compatible Stieltjes measure. Conversely, if \(t_0\in Z_x\cap[a,b)\), then for every \(r>0\) the measure \(r\delta_{t_0}\) can be added to the undetected part without changing the observed trajectory. Hence the clock is not unique. Only increments of \(g\) enter the Stieltjes measure and the differential equation, so replacing \(g\) by \(g+C\) does not change the dynamics.
\end{proof}

\begin{example}[A clock jump not detected by the state]
\label{ex:silent-clock-atom}
Let \(I=[0,1]\), fix \(\alpha>1\), and define the known vector field \(F:I\times\R\to\R\) by \(F(t,y):=2t-1/2\). For each \(r>0\), define the left-continuous nondecreasing Stieltjes derivator \(g_r:I\to\R\) by
\[
g_r(t):=
\begin{cases}
0, & 0\le t\le \frac14,\\
r, & \frac14<t\le \frac34,\\
r+1, & \frac34<t\le1,
\end{cases}
\]
and the state \(x:I\to\R\) by
\[
x(t):=
\begin{cases}
0, & 0\le t\le \frac34,\\
1, & \frac34<t\le1.
\end{cases}
\]
The jump points of the clock at \(1/4\) and \(3/4\) have masses \(r\) and \(1\), respectively. At the first atom, \(F(1/4,x(1/4))=0\) and \(\Delta^+x(1/4)=0\), so the clock jump of size \(r\) does not change the state. At the second, \(F(3/4,x(3/4))=1\) and \(\Delta^+x(3/4)=1\), in agreement with the atomic relation. Since the only nonzero state jump has size one across a unit clock jump, Theorem~\ref{thm:pure-jump-characterization} gives \([x]_{\alpha,g_r}=1\). Thus the same complete state trajectory is generated for every \(r>0\), while the total Stieltjes mass is \(r+1\). Thus the parameter \(r>0\) is an arbitrary measure component concentrated on the zero set along the trajectory. This gives an explicit instance of the nonuniqueness described by Theorem~\ref{thm:compatible-single-clocks} and shows that complete observation of the state trajectory cannot recover Stieltjes mass placed where the known dynamics vanish.
\end{example}

\subsection{Affine recovery formulas}

For the affine model of Sections~5--6, observations of the individual state jumps determine the clock jump sizes, the total Stieltjes mass, and the smallest compatible super-\(g\)-H\"older constant explicitly.

\begin{proposition}[Recovery formulas for affine dynamics]
Under the affine assumptions of Sections~5--6, let the nonzero state jump points be \(\tau_1<\cdots<\tau_N\). Set \(m_j:=\Delta^+g(\tau_j)\), let \(x_{j-1}\in\R^n\) denote the state immediately before the \(j\)-th jump, set \(x_j:=x(\tau_j+)\), \(\Delta x_j:=x_j-x_{j-1}\), and \(R_j:=\|x_j-x_*\|\). Then
\[
m_j=\frac{\|\Delta x_j\|}{\sigma\|x_{j-1}-x_*\|}
=\frac1{\sigma}\left(\frac{R_j}{R_{j-1}}-1\right),
\tag{7.15}
\]
so
\[
M=\frac1{\sigma}\sum_{j=1}^N\left(\frac{R_j}{R_{j-1}}-1\right),
\tag{7.16}
\]
and the smallest H\"older constant compatible with the recovered clock is
\[
L_{\min}=\max_{1\le j\le N}\frac{\sigma R_{j-1}}{m_j^{\alpha-1}}.
\tag{7.17}
\]
\end{proposition}

\begin{proof}
The affine jump relation is \(\Delta x_j=\sigma m_j(x_{j-1}-x_*)\), which gives (7.15). Summing the recovered masses gives (7.16), while the atomic seminorm identity of Theorem~\ref{thm:atomic-dynamics} gives (7.17).
\end{proof}

\begin{example}[Same endpoint and regularity bound, different Stieltjes clocks]
\label{ex:endpoint-nonid}
Let \(\sigma=1\), \(R_0=1\), \(\alpha=2\), and \(L=3\). A single atom \(m_1=1.25\) produces terminal multiplier \(2.25\). Two atoms \(m_1=m_2=0.5\) produce the same multiplier \((1.5)^2=2.25\). Both clocks are admissible under the same regularity bound: for the two-atom clock the first condition is \(0.5\ge1/3\), while the second is \(0.5\ge1.5/3\). Their total Stieltjes masses are nevertheless \(1.25\) and \(1.00\). Thus the endpoint determines only the multiplicative response \(\prod_{j=1}^{N}(1+\sigma m_j)\), not the number of jumps, the individual masses, or even their total mass. Observations of the individual state jumps contain strictly more clock information than terminal observations.
\end{example}

\section{Recovery of the clock for nonconvex Stieltjes differential inclusions}

For nonconvex set-valued dynamics, an observed state jump need not determine a unique clock jump size because several admissible velocities may lie on the same jump direction. We first determine the masses compatible with one observed jump and then characterize the Stieltjes measures compatible with the whole observed trajectory.

Let \(g:I\to\R\) be a Stieltjes derivator as in Section~2, and let \(\mathcal F:I\times\R^n\rightrightarrows\R^n\) be a nonempty-valued multifunction, not assumed convex. Consider
\[
x'_g(t)\in\mathcal F(t,x(t)).
\tag{8.1}
\]
A solution is a \(g\)-absolutely continuous \(x:I\to\R^n\) satisfying (8.1) \(\mu_g\)-almost everywhere. For \(m>0\) and \(A\subset\R^n\), scalar multiplication of a set is understood as \(mA:=\{mv:v\in A\}\); thus, at an atom \(\tau\in D_g\), the inclusion may be read as \(\Delta^+x(\tau)\in m_\tau\mathcal F(\tau,x(\tau))\). No general existence assertion is made here; all structural conclusions are conditional on the stated solution. Assume again \([x]_{\alpha,g}\le L\), \(L>0\), \(\alpha>1\). For the inverse interpretations below, the multifunction is regarded as known and the state jumps as resolved.

\begin{theorem}[Continuous and jump parts of a solution]
\label{thm:inclusion-atomic}
Let \(g:I\to\R\) be left-continuous and nondecreasing, let \(\alpha>1\) and \(L>0\), and let \(\mathcal F:I\times\R^n\rightrightarrows\R^n\) be nonempty-valued. Suppose \(x:I\to\R^n\) is a \(g\)-absolutely continuous solution of (8.1) and satisfies \([x]_{\alpha,g}\le L\). Then, for \(\mu_g^c\)-almost every \(t\in I\),
\[
0\in\mathcal F(t,x(t)).
\tag{8.2}
\]
At each jump point \(\tau\in D_g\), equivalently at each atom of \(\mu_g\), writing \(m_\tau:=\Delta^+g(\tau)>0\), there exists \(v_\tau\in\mathcal F(\tau,x(\tau))\) such that
\[
\Delta^+x(\tau)=m_\tau v_\tau,
\tag{8.3}
\]
and
\[
\|v_\tau\|\le Lm_\tau^{\alpha-1}.
\tag{8.4}
\]
\end{theorem}

\begin{proof}
Proposition~\ref{prop:continuous-extinction} gives \(x'_g=0\) \(\mu_g^c\)-almost everywhere, and hence (8.2). Since every \(\tau\in D_g\) has positive \(\mu_g\)-mass, the inclusion holds at each atom of \(\mu_g\). Setting \(v_\tau:=x'_g(\tau)\) gives (8.3), and the atomic H\"older estimate gives (8.4).
\end{proof}

Fix an observed state jump \(\tau\in D_x\), set \(z:=\Delta^+x(\tau)\ne0\), and write \(\xi:=x(\tau)\in\R^n\). Any compatible Stieltjes clock must have a positive jump at \(\tau\). A candidate mass \(m>0\) is locally compatible with the observation exactly when \(z/m\in\mathcal F(\tau,\xi)\).

\begin{definition}[Compatible jump sizes at one observed state jump]
\label{def:local-ambiguity}
For \(\tau\in I\), \(\xi\in\R^n\), and \(z\in\R^n\setminus\{0\}\), define
\[
\mathcal M_{\mathcal F}(\tau,\xi;z)
:=\left\{m>0:\frac zm\in\mathcal F(\tau,\xi)\right\}.
\tag{8.5}
\]
This set consists of all masses compatible with the observed jump; membership alone does not imply compatibility with the whole observed trajectory.
\end{definition}

For \(z\in\R^n\setminus\{0\}\), define the positive ray
\[
R_z:=\{\lambda z:\lambda>0\}.
\tag{8.6}
\]

The atomic inclusion relation can be written in terms of the positive ray generated by the observed jump.

\begin{theorem}[Characterization by the observed jump direction]
\label{thm:ray-intersection}
Let \(\mathcal F:I\times\R^n\rightrightarrows\R^n\) be nonempty-valued. Let \(\tau\in I\), \(\xi\in\R^n\), and \(z\in\R^n\setminus\{0\}\). Define \(\Psi:\mathcal F(\tau,\xi)\cap R_z\to(0,\infty)\) by
\[
\Psi(v):=\frac{\|z\|}{\|v\|}.
\]
Then \(\Psi\) is a bijection onto \(\mathcal M_{\mathcal F}(\tau,\xi;z)\). Thus the locally compatible jump sizes are determined exactly by the intersection of the velocity set with the positive ray through the observed jump.
\end{theorem}

\begin{proof}
If \(v\in\mathcal F(\tau,\xi)\cap R_z\), then \(v=\lambda z\) for a unique \(\lambda>0\). Setting \(m=1/\lambda=\|z\|/\|v\|\) gives \(z/m=v\in\mathcal F(\tau,\xi)\), so \(m\in\mathcal M_{\mathcal F}(\tau,\xi;z)\). Conversely, if \(m\in\mathcal M_{\mathcal F}(\tau,\xi;z)\), then \(v=z/m\) belongs to \(\mathcal F(\tau,\xi)\cap R_z\) and satisfies \(m=\|z\|/\|v\|\). The two constructions are inverse to each other.
\end{proof}

\begin{example}[Two compatible jump sizes in \(\R^2\)]
\label{ex:two-clock-candidates}
Let \(I=[0,1]\), let \(e_1=(1,0)\) and \(e_2=(0,1)\) be the standard basis vectors of \(\R^2\), and define the known nonconvex multifunction
\[
\mathcal F:I\times\R^2\rightrightarrows\R^2,
\qquad
\mathcal F(t,y):=\{e_1,2e_1,e_2\}
\]
for every \((t,y)\in I\times\R^2\). Fix \(\tau=1/2\in I\), \(\xi=(0,0)\in\R^2\), and \(z=e_1\in\R^2\). Then \(R_z=\{\lambda e_1:\lambda>0\}\) and \(\mathcal F(\tau,\xi)\cap R_z=\{e_1,2e_1\}\); the third velocity \(e_2\) does not contribute to the observed jump direction. By Theorem~\ref{thm:ray-intersection},
\[
\mathcal M_{\mathcal F}(\tau,\xi;e_1)=\left\{\frac12,1\right\}.
\tag{8.7}
\]
Thus the same observed state jump is compatible with two different jump sizes. The super-\(g\)-H\"older bound can remove some of these candidates.
\end{example}

The super-\(g\)-H\"older condition provides an additional filter. By the atomic H\"older bound, every compatible mass must also satisfy
\[
m\ge\left(\frac{\|z\|}{L}\right)^{1/\alpha}.
\tag{8.8}
\]
Thus the H\"older bound gives a lower bound for the jump size, equivalently an upper bound for the admissible velocity magnitude along the observed jump direction.

\begin{definition}[Compatible masses under the H\"older bound]
\label{def:filtered-local-ambiguity}
For \(\tau\in I\), \(\xi\in\R^n\), \(z\in\R^n\setminus\{0\}\), \(\alpha>1\), and \(L>0\), define
\[
\mathcal M_{\mathcal F}^{\alpha,L}(\tau,\xi;z)
:=
\mathcal M_{\mathcal F}(\tau,\xi;z)
\cap
\left[\left(\frac{\|z\|}{L}\right)^{1/\alpha},\infty\right).
\tag{8.9}
\]
Equivalently, the corresponding velocity must satisfy
\[
\|v\|\le L^{1/\alpha}\|z\|^{(\alpha-1)/\alpha}.
\tag{8.10}
\]
\end{definition}

The preceding characterization gives the following uniqueness criterion.

\begin{corollary}[Uniqueness under the H\"older bound]
\label{cor:regularity-assisted-identifiability}
Let \(\mathcal F:I\times\R^n\rightrightarrows\R^n\) be nonempty-valued, let \(\tau\in I\), \(\xi\in\R^n\), and \(z\in\R^n\setminus\{0\}\), and fix \(\alpha>1\) and \(L>0\). The jump size is locally unique under both the inclusion and the super-\(g\)-H\"older regularity information if and only if
\[
\operatorname{card}\!\left(
\mathcal F(\tau,\xi)\cap R_z\cap
\{v\in\R^n:\|v\|\le L^{1/\alpha}\|z\|^{(\alpha-1)/\alpha}\}
\right)=1.
\tag{8.11}
\]
\end{corollary}

\begin{proof}
By Theorem~\ref{thm:ray-intersection}, locally compatible jump sizes are in one-to-one correspondence with velocities in \(\mathcal F(\tau,\xi)\cap R_z\). Definition~\ref{def:filtered-local-ambiguity} restricts these velocities exactly to those satisfying the regularity bound in (8.10). Hence the jump size is unique under the H\"older bound if and only if the displayed intersection contains exactly one velocity.
\end{proof}

\begin{example}[The H\"older bound leaves one compatible mass]
To see how the super-\(g\)-H\"older constraint acts on the two jump-size candidates found above, continue with \(I=[0,1]\), the multifunction \(\mathcal F:I\times\R^2\rightrightarrows\R^2\), and the points \(\tau=1/2\), \(\xi=(0,0)\), and \(z=e_1\) from Example~\ref{ex:two-clock-candidates}. Take \(\alpha=2\) and \(L=2\). The regularity condition requires \(m\ge1/\sqrt2\). Since the compatible-mass set is \(\{1/2,1\}\),
\[
\mathcal M_{\mathcal F}^{2,2}(\tau,\xi;e_1)=\{1\}.
\tag{8.12}
\]
Thus the super-\(g\)-H\"older constraint eliminates the smaller jump-size candidate and makes the mass locally identifiable, even though the known velocity set remains nonconvex and contains velocities in two independent directions.
\end{example}

For a multifunction \(\mathcal F:I\times\R^n\rightrightarrows\R^n\), define its pointwise closed convexification at each \((t,\xi)\in I\times\R^n\) by
\[
(\overline{\operatorname{co}}\mathcal F)(t,\xi)
:=
\overline{\operatorname{co}}\bigl(\mathcal F(t,\xi)\bigr).
\]

\begin{theorem}[Convexification can enlarge the set of compatible jump sizes]
\label{thm:convexification}
Let \(\mathcal F:I\times\R^n\rightrightarrows\R^n\) be nonempty-valued, let \(\tau\in I\), \(\xi\in\R^n\), and \(z\in\R^n\setminus\{0\}\). For every \(\alpha>1\) and \(L>0\),
\[
\mathcal M_{\mathcal F}(\tau,\xi;z)
\subseteq
\mathcal M_{\overline{\operatorname{co}}\mathcal F}(\tau,\xi;z),
\tag{8.13}
\]
and
\[
\mathcal M_{\mathcal F}^{\alpha,L}(\tau,\xi;z)
\subseteq
\mathcal M_{\overline{\operatorname{co}}\mathcal F}^{\alpha,L}(\tau,\xi;z).
\tag{8.14}
\]
\end{theorem}

\begin{proof}
Both statements follow from
\(\mathcal F(\tau,\xi)\subseteq(\overline{\operatorname{co}}\mathcal F)(\tau,\xi)\).
\end{proof}

The inclusions may remain strict even after the regularity filter is applied.

Return to the two-dimensional multifunction of Example~\ref{ex:two-clock-candidates}. Its pointwise closed convexification is the triangle
\[
\bigl(\overline{\operatorname{co}}\mathcal F\bigr)(t,y)
=\operatorname{co}\{e_1,2e_1,e_2\},
\qquad (t,y)\in I\times\R^2.
\]
At \(\tau=1/2\), \(\xi=(0,0)\), and \(z=e_1\), the intersection with the observed positive ray is \(\{\lambda e_1:1\le\lambda\le2\}\). Hence the compatible-mass set is \([1/2,1]\), while under \(\alpha=2\) and \(L=2\) the set satisfying the H\"older bound is \([1/\sqrt2,1]\). Thus convexification turns the two discrete jump-size candidates into a continuum and destroys the uniqueness under the H\"older bound obtained in the preceding example.

For a zero observed state jump at \(t\in[a,b)\), an arbitrary positive clock jump size is compatible with zero state change exactly when
\[
0\in\mathcal F(t,x(t)).
\tag{8.15}
\]
Similarly, nonatomic Stieltjes mass that leaves the state unchanged requires zero to be available as a velocity selection \(\mu_g^c\)-almost everywhere. Define
\[
Z_{\mathcal F}(x):=\{t\in I:0\in\mathcal F(t,x(t))\}.
\tag{8.16}
\]
Additional Stieltjes mass that does not change the observed state can occur only on this set.

The same atomic estimate also bounds the number of nonzero state jumps without any convexity assumption.

\begin{corollary}[Lower bound for nonzero state jumps and their number]
\label{cor:active-velocity-gap}
Let \(x:I\to\R^n\) satisfy the assumptions of Theorem~\ref{thm:inclusion-atomic}. Assume there exists a constant \(c>0\) such that, for every \(t\in I\) and every nonzero \(v\in\mathcal F(t,x(t))\), one has \(\|v\|\ge c\). Then every nonzero state jump point \(\tau\in D_g\) satisfies
\[
m_\tau\ge\left(\frac{c}{L}\right)^{1/(\alpha-1)},
\tag{8.17}
\]
and, if \(M>0\) and \(\mu_g([a,b))\le M\), then
\[
\operatorname{card}\{\tau\in D_g:\Delta^+x(\tau)\ne0\}
\le
\left\lfloor M\left(\frac{L}{c}\right)^{1/(\alpha-1)}\right\rfloor.
\tag{8.18}
\]
No convexity assumption is required.
\end{corollary}

\begin{proof}
At a nonzero state jump point, Theorem~\ref{thm:inclusion-atomic} gives \(\|v_\tau\|\le Lm_\tau^{\alpha-1}\). Combining this with \(\|v_\tau\|\ge c\) yields the lower bound in (8.17). Summing these minimum masses over all nonzero state jump points gives (8.18).
\end{proof}

The preceding results combine to give a characterization of all Stieltjes measures compatible with an observed trajectory and its individual state jumps.

\begin{theorem}[All Stieltjes measures compatible with a trajectory]
\label{thm:compatible-inclusion-clocks}
Let \(x:I\to\R^n\) be an observed trajectory, including its individual state jumps, satisfying the assumptions of Theorem~\ref{thm:inclusion-atomic}. A finite Stieltjes measure \(\mu\), with associated clock \(g_\mu(t):=C+\mu([a,t))\) for an arbitrary \(C\in\R\), generates the same observed trajectory and is compatible with the differential inclusion and the bound \([x]_{\alpha,g_\mu}\le L\) if and only if
\[
\mu=\sum_{\tau\in D_x}m_\tau\delta_\tau+\nu,
\]
where, for every \(\tau\in D_x\),
\[
m_\tau\in\mathcal M_{\mathcal F}^{\alpha,L}\bigl(\tau,x(\tau);\Delta^+x(\tau)\bigr),
\]
the selected masses satisfy
\[
\sum_{\tau\in D_x}m_\tau<\infty,
\]
and \(\nu\) is a finite nonnegative Borel measure concentrated on \(Z_{\mathcal F}(x)\setminus D_x\).
\end{theorem}

\begin{proof}
For a compatible clock, Theorem~\ref{thm:inclusion-atomic} and Definition~\ref{def:filtered-local-ambiguity} force each observed-jump mass to satisfy the H\"older bound. Finiteness of the Stieltjes measure gives \(\sum_{\tau\in D_x}m_\tau<\infty\). At points outside \(D_x\), any additional Stieltjes mass can leave the state unchanged only where zero velocity is admissible, hence it is concentrated on \(Z_{\mathcal F}(x)\setminus D_x\).

Conversely, choose masses and \(\nu\) as stated. For each \(\tau\in D_x\), select \(v_\tau\in\mathcal F(\tau,x(\tau))\) with \(\Delta^+x(\tau)=m_\tau v_\tau\); the H\"older bound in the definition gives \(\|\Delta^+x(\tau)\|\le Lm_\tau^\alpha\). On the \(\nu\)-part choose the zero velocity. Let \(\mu\) be the resulting measure and for an arbitrary \(C\in\R\), define \(\widetilde g(t):=C+\mu([a,t))\). Then the inclusion holds with respect to \(\widetilde g\). Moreover, for \(s<t\),
\[
\|x(t)-x(s)\|
\le
L\sum_{\tau\in D_x\cap[s,t)}m_\tau^\alpha
\le
L\left(\sum_{\tau\in D_x\cap[s,t)}m_\tau\right)^\alpha
\le
L\mu([s,t))^\alpha.
\]
Thus the same trajectory satisfies the required super-\(\widetilde g\)-H\"older bound. Proposition~\ref{prop:holder-implies-gac} then gives \(\widetilde g\)-absolute continuity, completing the converse.
\end{proof}

\begin{corollary}[Uniqueness criterion for the Stieltjes clock]
\label{cor:complete-inclusion-recovery}
The compatible Stieltjes clock is uniquely determined up to an additive constant if and only if both
\[
Z_{\mathcal F}(x)\cap([a,b)\setminus D_x)=\varnothing
\]
and, for every \(\tau\in D_x\),
\[
\operatorname{card}\!\left(\mathcal M_{\mathcal F}^{\alpha,L}\bigl(\tau,x(\tau);\Delta^+x(\tau)\bigr)\right)=1
\]
hold.
\end{corollary}

\begin{proof}
Theorem~\ref{thm:compatible-inclusion-clocks} shows that uniqueness requires and is implied by two independent facts: there is exactly one admissible mass at each observed jump, and there is no location outside the observed jump set at which additional Stieltjes mass can be placed without changing the state. These are precisely the two stated conditions. The remaining additive constant does not affect the dynamics, which depend only on increments of the clock.
\end{proof}

\section{Discussion and conclusions}

H\"older regularity above exponent one has a different structure when increments are measured by a Stieltjes clock. For an ordinary continuous clock, exponent \(\alpha>1\) forces constancy. For a left-continuous nondecreasing Stieltjes clock, the nonatomic part still cannot support nonconstant variation, but jumps can. Every finite-dimensional super-\(g\)-H\"older function therefore admits a pure-jump representation, and its \(g\)-H\"older seminorm is determined exactly by its normalized jump amplitudes. After fixing the initial value, this identifies the function space linearly and isometrically with \(\ell^\infty(D_g;\R^d)\); without fixing the initial value, the corresponding space is isometric to \(\R^d\oplus_1\ell^\infty(D_g;\R^d)\). The Banach structure, the criterion that separability is equivalent to finiteness of \(D_g\), the exact variation formula, and the finite-jump approximation results all follow from this atomic representation.

Uniform approximation by finite-jump functions is available for every super-\(g\)-H\"older function, whereas approximation in the natural super-\(g\)-H\"older norm occurs exactly for the subspace \(c_0(D_g;\R^d)\) of the sequence space. Moreover, increasing the H\"older exponent from \(\alpha\) to \(\beta>\alpha\) produces a compact embedding whose finite-rank truncation error is determined exactly by the ordered clock jump sizes.

For Stieltjes differential equations, the same atomic structure restricts the possible state changes. If the vector field at nonzero state jumps is uniformly bounded away from zero, each state change requires a positive minimum clock increment, so only finitely many such jumps can occur under a finite total Stieltjes mass. In the affine model, the minimal admissible jump sizes satisfy an explicit recurrence. Exactly \(N\) nonzero state jumps are possible if and only if \(M\ge S_N\), so the threshold sequence \((S_N)\) determines the maximum number of jumps.

For a fixed feasible number of jumps, the terminal distance has exact minimum and maximum values. The minimum keeps the early jump sizes at their smallest admissible values and assigns the remaining mass to the final jump. At the maximum, an initial block has equal jump sizes and the remaining sizes are determined recursively by saturated constraints. Optimizing over the feasible values of \(N\) gives the global extrema. For the same total clock increment \(M\), every response corresponding to admissible clock-jump sizes remains strictly below the corresponding continuous-clock response \(R_0e^{\sigma M}\).

For known single-valued dynamics, the inverse problem admits a complete description. Observations of the individual state jumps determine the corresponding atomic part of the Stieltjes clock. If \(Z_x=\{t\in I:F(t,x(t))=0\}\) and \(\mu_{\rm obs}\) denotes the atomic measure reconstructed from the observed nonzero state jumps, then every compatible Stieltjes measure has the form
\[
\mu_g=\mu_{\rm obs}+\nu,
\]
where \(\nu\) is an arbitrary finite nonnegative measure concentrated on \(Z_x\cap[a,b)\). Hence the unresolved part of the clock is characterized exactly by the zero set \(Z_x\) of the vector field along the trajectory. In particular, the Stieltjes clock is uniquely determined up to an additive constant if and only if \(Z_x\cap[a,b)=\varnothing\). This characterization also shows why observations of the individual state jumps contain substantially more information than terminal observations, which may correspond to different numbers of jumps and different total Stieltjes masses.

For nonconvex differential inclusions, the same observed state jump may be compatible with more than one clock jump. The super-\(g\)-H\"older condition reduces these possibilities and can restore uniqueness in some cases. Additional clock variation may remain invisible where zero velocity is allowed. Thus recovery of the Stieltjes clock depends both on the ambiguity at the observed jumps and on whether unobserved clock variation can occur elsewhere. Convexification may enlarge the set of compatible clock jumps and can therefore reduce identifiability.

For \(\alpha>1\), the jump structure of the Stieltjes clock provides a common description of the function space, the dynamics, and the recovery of the clock from observations. This common atomic structure is the main link between the regularity, extremal, and inverse results of the paper.

\end{document}